\documentclass[12pt,oneside]{amsart}
\usepackage{amssymb}
\usepackage{amsmath}
\usepackage{amsthm}
\usepackage{amscd}

\usepackage{longtable}

\theoremstyle{plain}
\newtheorem{theorem}{Theorem}
\newtheorem{lemma}{Lemma}

\newtheorem{proposition}{Proposition}

\theoremstyle{definition}

\theoremstyle{remark}
\newtheorem{remark}{Remark}

\begin{document}
\title
[Polarizations and Convergence of holomorphic sections]
{Polarizations and Convergences of holomorphic sections on the tangent bundle of a Bohr-Sommerfeld Lagrangian submanifold}
\author{Yusaku Tiba}
\date{}

\begin{abstract}
Let $(M, \omega)$ be a K\"ahler manifold and let $(L, \nabla)$ be a prequantum line bundle over $M$.  
Let $X \subset M$ be a Bohr-Sommerfeld Lagrangian submanifold of $(L, \nabla)$.  
In this paper, we study an asymptotic behaviour of holomorphic sections of $L^k$ as $k \to \infty$.  
Our first result shows that the $L^2$-norm of sections of $L^k$ are bounded below around $X$ 
if these sections converge on $X$ under a suitable trivialization of $L^k$.   
Since $X$ is a Lagrangian submanifold, we consider that a neighborhood of $X$ is embedded in the tangent bundle $TX$.  
Let $\Psi_{k}: TX \to TX$ be a multiplication by $\frac{1}{\sqrt{k}}$ in the fibers.  
The pullback of the K\"ahler polarization by $\Psi_k$ converges to the real polarization, whose leaves are fibers of $TX$, as $k \to \infty$.  
Let $(f_k)_{k \in \mathbb{N}}$ be holomorphic sections of $L^k$ near $X$.  
By trivializing $L^k$, we consider $f_k$ as a function.  
In our second result, we show that $\Psi^* f_k$ 
converges to a fiberwise constant function on $TX$ as $k \to \infty$ under some condition on Sobolev norms of $f_k$.  
\end{abstract}

\maketitle

\subjclass{{\bf 2010 Mathematics Subject Classification.} 32A25, 53D12}
%\keywords{{\bf Keywords.}Pseudoconvex domain \and Plurisubharmonic function}

%\classification{32U10, 32L10}

\section{Introduction}\label{section:1}
Any continuous function on the unit circle $S^1 \subset \mathbb{C}$ can be approximated uniformly by polynomials in $e^{\sqrt{-1} m \theta}$ ($m \in \mathbb{Z}$).  
If we replace $e^{\sqrt{-1}m\theta}$ with $z^m$, continuous functions can be approximated by holomorphic functions on a neighborhood of $S^1$.   
Such an approximation theorem is generalized to higher-dimensional cases.  
Let $X \subset \mathbb{C}^n$ be a totally real compact submanifold.  
Let $f$ be a continuous function on $X$.  
Then $f$ can be approximated uniformly by holomorphic functions defined on a neighborhood of $X$ (see Theorem~1 of \cite{Ran-Siu}).   

We consider a corresponding theorem for holomorphic sections of a line bundle.  
Let $(M, \omega)$ be a K\"ahler manifold of complex dimension $n$ with a complex structure $J$.  
Let $(L, \nabla, h)$ be a holomorphic prequantum line bundle over $M$ with a Hermitian metric $h$.  
That is, $\nabla$ is the Chern connection of $(L, h)$ whose curvature form is $F^{\nabla} = -2\sqrt{-1} \omega$.   
%As $F^{\nabla}$ is a $(1, 1)$-form, $L$ is a holomorphic line bundle and $\nabla$ is the Chern connection.  
We denote the $k$-th tensor power of $L$ by $L^k$.   
Let $g_1, g_2$ be sections of $L^k$ on $M$.  
We define the integral product $(g_1, g_2)_{h^k} = \int_{M} \langle g_1, g_2 \rangle_{h^k} \frac{\omega^n}{n!}$ where 
$\langle g_1, g_2 \rangle_{h^k}$ is the pointwise scalar product.  
Put $\|g_{1}\|^2_{h^k} = (g_1, g_1)_{h^{k}}$ and $|g_1|^2_{h^k} = \langle g_1, g_1 \rangle_{h^k}$.  
We denote by $L^2(M, L^k)$ the Hilbert space of $L^2$-integrable sections of $L^k$, and 
denote by $\Gamma(M, L^k)$ the space of holomorphic sections of $L^k$ over $M$.  
Let $X \subset M$ be a compact Lagrangian submanifold of $(M, \omega)$.  
We call $X$ a Bohr-Sommerfeld Lagrangian submanifold if the restricted line bundle $(L|_{X}, \nabla|_{X})$ is a trivial line bundle, that is, there exists a smooth section $\zeta$ of $L|_{X}$ on $X$ such that $\nabla|_{X} \zeta = 0$ and that $|\zeta|_{h} = 1$.    
Since any continuous function can be approximated uniformly by a smooth one on $X$, 
we will only consider smooth functions.  
Let $f \in C^{\infty}(X)$.  
If $M$ is compact, 
there exists a sequence of holomorphic sections $(f_{k})_{k \in \mathbb{N}}$ ($f_k \in \Gamma(M, L^k)$) such that $f_{k}/\zeta^k$ approximate $f$ uniformly on $X$.  
A straightforward way to construct such a sequence is to take the Bergman projection of a distribution $f \zeta^k dv_{X}$.  
Here $dv_{X}$ is the volume density on $X$ induced by the Riemannian metric $g_{\omega}(\cdot, \cdot) = \omega(\cdot, J\cdot)$.  
Let $K_{k}(z, w)$ ($z, w \in M$) be the Bergman kernel of $L^k$.  
$K_{k}(z, w)$ is the Schwartz kernel of the orthogonal projection from $L^2(M, L^k)$ to $L^{2}(M, L^k)\cap \Gamma(M, L^k)$.  
Put 
\[
s_{f, k}(z) = \left(\frac{\pi}{2k}\right)^{n/2} \int_{X} K_k(z, x) f(x) \zeta^k(x) dv_{X}(x).  
\]
$\left(\frac{2k}{\pi}\right)^{n/2}s_{f, k}$ represents the quantizations of $(X, f dv_{X})$ at Planck's constant $1/k$.  
An asymptotic behaviour of $s_{f, k}$ is extensively studied (e.g., \cite{Bor-Pau-Uri}, \cite{Deb-Pao}, \cite{Pao}, \cite{Ioo}).  
It is known that 
$\lim_{k \to \infty}\frac{s_{f, k}|_{X}}{\zeta^k} = f$ uniformly on $X$, and that 
\[
\lim_{k \to \infty}\left(\frac{2k}{\pi} \right)^{n/2}\|s_{f, k}\|^2_{h^k} = \int_{X} |f|^2 dv_{X}
\]
(see Section~\ref{section:7}).  
The following proposition shows that $(s_{f, k})_{k \in \mathbb{N}}$ is a minimum $L^2$-norm sequence which approximates $f$.   
%Let $(f_k)_{k \in \mathbb{N}}$ ($f_k \in \Gamma(M, L^k)$) be a sequence of holomorphic sections such that $f_k/\zeta^k$ converges to $f$ on $X$.  
%Since a holomorphic section vanishes identically on $M$ if it vanishes on $X$, it would be reasonable to expect that $f$ influences an asymptotic behaviour of $f_k$ on $M$.  
%In this context, the following holds:  
\begin{proposition}\label{proposition:0}
Assume that $M$ is compact.  
Let $f$ be a smooth function on $X$.  
Let $f_k \in \Gamma(M, L^k)$ ($k \in \mathbb{N}$) be a sequence of holomorphic sections such that 
$\lim_{k \to \infty}\frac{f_k|_{X}}{\zeta^k}=f$ in the $L^2$-norm on $X$ with respect to the measure $dv_X$.  
Then 
\begin{equation}\label{equation:0}
\liminf_{k \to \infty} \left(\frac{2k}{\pi} \right)^{n/2} \|f_k\|^2_{h^k}\geq \int_{X}|f|^2 dv_{X}, 
\end{equation}
and the equation holds if and only if $\liminf_{k \to \infty} k^{n/2}\|f_k-s_{f, k}\|^2_{h^k} = 0$.  
\end{proposition}
The proof of the above proposition is simple and will be given in Section~\ref{section:2}.  
In this paper, we investigate Proposition~\ref{proposition:0} further.     
The Bergman kernel decays rapidly off the diagonal as $k \to \infty$.  
Hence $s_{f, k}$ becomes very small away from $X$ and is concentrated around $X$. 
Our interest lies in what happens to holomorphic sections around $X$.  
The first main result (Theorem~\ref{theorem:2}) shows that the inequality (\ref{equation:0}) holds \textit{micro-locally} around $X$ (see Section~\ref{section:4}).

The second main result (Theorem~\ref{theorem:3}) is motivated by the geometric quantization.   
A polarization on a symplectic manifold $(M, \omega)$ is a Lagrangian and integrable subbundle $\mathcal{P}$ of $TM\otimes \mathbb{C}$, that is, $\mathcal{P}$ satisfies $\omega(\mathcal{P}, \mathcal{P}) = 0$, $\mathrm{rank}\, \mathcal{P} = \frac{1}{2}\dim_{\mathbb{R}} M$, and $[\mathcal{P}, \mathcal{P}] \subset \mathcal{P}$.   
We call $\mathcal{P}$ a real polarization if $\mathcal{P}$ is given by the complexification of a subbundle of $TM$.  
By completing the space of $\mathcal{P}$-parallel sections of prequantum line bundles in some sense, we get a Hilbert space. 
Geometric quantization associates $(M, \omega)$ to this Hilbert space. 
For example, if $(M, \omega)$ is a K\"ahler manifold with a complex structure $J$, 
there exists the K\"ahler polarization $P_{J} = T^{(0, 1)}M$, and the Hilbert space $\Gamma(M, L^k) \cap L^2(M, L^k)$ is regarded as the quantum phase space of $X$ with the Planck constant $h=1/k$.  
Letting $k$ tend to infinity corresponds to letting $h$ tend to 0, which is referred to as the \textit{semiclassical limit}, and asymptotic results as $k \to \infty$ expected to recover the laws of classical mechanics.  
In some cases, 
real polarizations are interpreted as degenerate K\"ahler polarizations under deformations of complex structures, and the convergence of these Hilbert spaces is studied (e.g., \cite{Bai-Flo-Mou-Nun}, \cite{Bai-Mou-Nun}, \cite{Hat}, \cite{Hat-Yam}).   
Now we explain our setup in detail.  
Let $(M, \omega)$ be a K\"ahler manifold of complex dimension $n$, $(L, \nabla, h)$ be a prequantum line bundle, and $X \subset M$ be a Bohr-Sommerfeld Lagrangian submanifold.  
Let $U \subset M$ be a sufficiently small neighborhood of $X$.  
Since $X$ is a Lagrangian submanifold, we can identify the normal bundle $JTX$ of $X$ with its tangent bundle $TX$ via $J$.  
By taking a tubular neighborhood of $X$, we may assume $U \subset TX$.  
We denote by $\Pi:TX \to X$ the projection map.  
There exists an isomorphism from $TX$ to $T^*X$ that maps $\nu \in TX$ to $g_{\omega}(\cdot, \nu) \in T^{*}X$.  
Then, we have a symplectic form 
$\omega_0$ on $TX$, which is the pull back of the canonical symplectic form on $T^*X$ by this isomorphism.  
Let $\mathcal{P}_J = T^{(0, 1)} U$ be the K\"ahler polarization of $(U, \omega)$.  
Let $\mathcal{P}_0 = \mathrm{Ker}\, d\Pi \subset TTX$, that is, the vertical polarization of $(TX, \omega_0)$.   
%The quantum Hilbert space of $\mathcal{P}_J$ consists of $L^2$-holomorphic sections of $L^k$.  
We define $\Psi_{k}: TX \to TX$ as the multiplication by $\frac{1}{\sqrt{k}}$ in the fibers of $TX$.  
Since $\omega = -\omega_0$ on $X$, we have 
$\lim_{k \to \infty}\sqrt{k}\Psi_k^*\omega  = -\omega_0$.  
Let $E$ be a vector field on $X$ and let $E^v$ be the vertical lift of $E$ on $TX$ (see Section~\ref{section:5} for the definition of the vertical lift).  
We have $\frac{1}{\sqrt{k}}\left((\Psi_{k*})^{-1}(E^v + \sqrt{-1}JE^{v})\right)(\nu) = E^v(\nu)$ for any $\nu \in TX$ (see Section~\ref{section:5}).  
This implies that the pullback of $(\sqrt{k} \omega, \mathcal{P}_{J})$ by $\Psi_k$ \textit{converges} to $(-\omega_0, \mathcal{P}_0)$ as $k \to \infty$.  
Let $(f_k)_{k \in \mathbb{N}}$ be holomorphic sections of $L^k$ near $X$.  
By trivializing $L^k$, we consider $f_k$ as a function.  
Then Theorem~\ref{theorem:3} shows that the pullback of $f_k$ by $\Psi_k$ 
converges to a fiberwise constant function on $TX$ as $k \to \infty$ under some condition on Sobolev norms of $f_k$.  
See Section~\ref{section:6} for precise statement of Theorem~\ref{theorem:3} and this condition.  
We also explain a relation between Theorem~\ref{theorem:3} and $(s_{f, k})_{k \in \mathbb{N}}$.  
Let  $s$ be a smooth section of $L$ on $U$ such that $s$ is an almost holomorphic extension of the $\nabla$-parallel section $\zeta \in C^{\infty}(X, L|_{X})$ (see Section~\ref{section:2}).   
In Section~\ref{section:7}, we prove the asymptotic expansion of $(s_{f, k})_{k \in \mathbb{N}}$, and prove that $\lim_{k \to \infty}\Psi_{k}^{*}(s_{f, k}/s^k)=\Pi^*f$ uniformly on any compact subset of $TX$ as in Theorem~\ref{theorem:3}.  
%This shows that the pullback of $s_{f, k}$ by $\Psi_{k}$ converges to a fiberwise constant function on $TX$ as $k \to \infty$ as in Theorem~\ref{theorem:3}.  
Although the asymptotic expansion formula of $(s_{f, k})_{k \in \mathbb{N}}$ is already known, we provide a proof of this formula since the derivation and our setup are somewhat different from previous works.  
\medskip

We introduce some notation.  
We write $f \lesssim g$ to mean that $|f| \leq c |g|$ for some constant $c>0$ which does not depend on $k$.  
We denote $\int_{V}|g|^2_{h^k}\frac{\omega^n}{n!}$ by $\|g\|^2_{V, h^k}$ for $L^2$-integrable section $g$ of $L^k$ on an open subset $V \subset M$.  
Let $\alpha$ be a form (resp.\!\! an $L^k$-valued form). 
We denote by $|\alpha|_{\omega}$ (resp. \!\!$|\alpha|_{h^k, \omega}$) the pointwise norm of $\alpha$ with respect to $g_{\omega}$.  
\medskip 

{\it Acknowledgment.}
The author would like to thank Kota Hattori for valuable advice.  
This work was supported by the 
Grant-in-Aid for Scientific Research (KAKENHI No.\! 21K03266).

\section{Asymptotic approximation by holomorphic sections}\label{section:2}
Let $(M, \omega)$ be a K\"ahler manifold of complex dimension $n$ and let $(L, \nabla, h)$ be a holomorphic prequantum line bundle on $M$.  
Let $X$ be a Bohr-Sommerfeld Lagrangian submanifolds in $M$ 
and let $\zeta$ be a smooth section on $X$ such that $\nabla|_{X} \zeta = 0$ and that $|\zeta| = 1$.   
Let $U \subset M$ be a sufficiently small tubular neighborhood of $X$.  
Because $\zeta$ is a non-vanishing smooth section of $L|_{X}$,  we have $L|_U$ is trivial smooth line bundle.  
The Oka principle shows that $L|_U$ is trivial holomorphic line bundle. 
We take a non-vanishing holomorphic section $s_0 \in \Gamma(U, L)$.   
Take an almost holomorphic extension $\xi \in C^{\infty}(U)$ of $\zeta/s_0$, that is, $\xi = \zeta/s_0$ on $X$ and $\overline{\partial} \xi$ vanishes to any order on $X$.  
Put $s = \xi s_0 \in C^{\infty}(U, L)$.  
Then $\overline{\partial} s$ vanishes any order on $X$.  Furthermore, by Proposition~1 of \cite{Tib},   it follows that $\nabla s = 0$, $\log |s|^2 = 0$ and $d \log |s|^2 = 0$ on $X$.  
Let $p \in X$ and let $V$ be a sufficiently small neighborhood of $p$ in $M$. 
Let $E_j$ ($j = 1, \ldots, n$) be vector fields on $X \cap V$ such that $\{(E_1)_q, \ldots, (E_n)_q\}$ is an orthonormal basis of $T_{q}X$ at any $q \in X \cap V$.    
Let $J$ be the complex structure of $M$ 
and let $g_{\omega}(\cdot, \cdot) = \omega(\cdot, J\cdot)$ be the Riemannian metric induced by $\omega$.  
Since $X$ is a Lagrangian submanifold, the subspace $T_qX \subset T_qM$ is orthogonal to $\langle (JE_1)_{q}, \ldots, (JE_n)_q \rangle$ for $q \in X \cap V$.  
We take a smooth coordinate $(x, y) = (x_1, \ldots, x_n, y_1, \ldots, y_n)$ on $V$ such that 
$(0, \ldots, 0)$ corresponds to $p$, $(x_1, \ldots, x_n, 0, \ldots, 0)$ corresponds to a point in $X \cap V$, and $(x_1, \ldots, x_n, y_1, \ldots, y_n)$ corresponds to $\exp_{(x, 0)} (y_1 J E_1 + \cdots + y_n J E_{n}) \in M$.  
Here $\exp$ is the exponential map of the Riemannian manifold $(M, g_{\omega})$.  
Put $\varphi_0 = -\log |s_0|^2_{h}$ and $\varphi = -\log |s|^2_{h}= -\log |\xi|^2 + \varphi_0$.  
Then $\omega = \frac{i}{2}F^{\nabla}= \frac{i}{2}\partial \overline{\partial} \varphi_0 = \frac{i}{2}\partial \overline{\partial} (\varphi + \log |\xi|^2)$.  
Note that $\log |s|^2_{h}$ vanishes to order two and  
$\partial \overline{\partial} \log |\xi|^2$ vanishes to any order on $X$.  
Since $\omega(\frac{\partial}{\partial y_i}, J \frac{\partial}{\partial y_j}) = \delta_{ij}$ on $X$, 
we have $\varphi (x, y) = 2|y|^2 + O(|y|^3)$ on $V$ (see the first part of Section~4 of \cite{Tib}). 
Let $dv_X$ be the volume density on $X$ induced by the Riemannian metric $g_{\omega}$.  
Using the local coordinate system $(x, y)$, we consider $V$ as an open subset of $(X \cap V) \times \mathbb{R}^n$.  
Let $dv_Xdy_1 \cdots dy_n$ be the product measure on $(X \cap V) \times \mathbb{R}^n$.  
We have 
$\frac{\omega^n}{n!} = (1 + O(|y|))dv_X dy_1\cdots dy_n$ around $X \cap V$. 
Denote by $d_{X}(z)$ the distance from $z \in M$ to $X$ with respect to $g_{\omega}$.  
Let $f \in C^{\infty}(X)$.   
Now we introduce two different procedures to approximate smooth sections $fs^{k}$ ($k \in \mathbb{N}$) of $L^k$ on $X$ asymptotically by holomorphic sections.  \vspace{3mm}\\
$(\mathrm{A})$  
Assume that $M$ is either a compact manifold or a Stein manifold.  
If $M$ is a Stein manifold, we also assume that the Ricci form $\mathrm{Ric}(\omega)$ satisfies $\mathrm{Ric}(\omega) \geq -C \omega$ on $M$ for some $C > 0$.  
Let $\tilde{f} \in C^{\infty}(U)$ be an almost holomorphic extension of $f$, that is, $f = \tilde{f}$ on $X$ and $\overline{\partial} \tilde{f}$ vanishes to any order on $X$.  
Let $\kappa \in C^{\infty}_0(U)$ be a smooth function such that $\kappa = 1$ on a neighborhood of $X$.    
Put $\alpha_k = \kappa \tilde{f} s^k \in C^{\infty}(M, L^k)$.  
Then 
\[
\overline{\partial} \alpha_{k} 
= \overline{\partial} (\kappa \widetilde{f} \xi^k) s_0^k 
=  (\widetilde{f} \overline{\partial} \kappa  + \kappa \overline{\partial} \widetilde{f}
+ k \kappa \widetilde{f} \frac{\overline{\partial} \xi}{\xi} )s^k.   
\]
Since $U$ is sufficiently small, we may assume that $\varphi(z) \geq d_{X}(z)^2$ on $U$.  
Hence we have 
\[
\|\overline{\partial} \alpha_k\|^2_{h^k} \lesssim 
\int\left( |\widetilde{f}\overline{\partial} \kappa|^2 + |\kappa \overline{\partial} \widetilde{f}|^2 
+ k^2 |\widetilde{f} \kappa \overline{\partial} \xi|^2\right) e^{-kd_X(z)^2}\frac{\omega^n}{n!} = O(k^{-m})  
\]
for any $m \in \mathbb{N}$ 
since $\int_{\mathbb{R}^n} |y|^{2m-n}e^{-k|y|^2}dy_1 \cdots dy_n = O(k^{-m})$.  
By regarding $\alpha_k$ as a section of $T^{(n, 0)}U \otimes L^k$-valued $(n, 0)$-form, H\"ormander $L^2$-estimate shows that there exists 
$\beta_k \in C^{\infty}(M, L^k)$ such that  
$\overline{\partial} \beta_k = \alpha_k$ and that $\|\beta_k\|_{h^k} \lesssim O(k^{-m})$ for large $k$ (see e.g., Chapter~VIII, Theorem~6.1 of \cite{Dem}, Corollary~5.3 of \cite{Dem2}).   
\begin{proposition}\label{proposition:-2}
Put $F_{k} = \alpha_k - \beta_k \in \Gamma(M, L^k)$.  
Then $\lim_{k \to \infty}\frac{F_k|_X}{\zeta^k} = f$ uniformly on $X$.   
\end{proposition}
\begin{proof}
Let $b_k \in C^{\infty}(U)$ such that $\beta_k = b_k s^k$. 
Let $p \in X$.  
We take an open neighborhood $V$ of $p$ and a smooth coordinate $(x, y)$ on $V$ as described above.  
Let $B(p, r) \subset V$ be the Euclidean ball of center $p \in X$ and radius $r$ with respect to the coordinate $(x, y)$.  
By Lemma~15.1.8 of \cite{Hor}, we have 
\[
|b_k(p)|^2 = |b_k(p)\xi(p)^k|^2 e^{-k\varphi_0(p)} \lesssim \frac{1}{k^2}  \sup_{B(p, 1/k)} |\overline{\partial} (b_k \xi^k)|^2_{\omega} e^{-k\varphi_0(p)} + k^{2n} \int_{B(p, 1/k)} |b_k \xi^k|^2e^{-k\varphi_0(p)} \frac{\omega^n}{n!}.   
\] 
Here the first equality holds since $|\xi(p)|^2e^{-\varphi_0(p)} = e^{-\varphi(p)} = 1$.  
Let $c > 0$ a positive number which is larger than $C^1$-norm of $\varphi_0$ on a neighborhood of $X$.  
Then $k|\varphi_0(z) - \varphi_0(p)| \leq c$ for $z \in B(p, 1/k)$.  
For any $m \in \mathbb{N}$, we obtain 
\begin{align*} 
& \sup_{B(p, 1/k)}|\overline{\partial} (b_k\xi^k)|^2_{\omega} e^{-k\varphi_0(p)}
=  \sup_{B(p, 1/k)}|\overline{\partial} (\kappa \widetilde{f} \xi^k)|^2_{\omega} e^{-k\varphi_0(p)} 
\lesssim \sup_{B(p, 1/k)} |\overline{\partial} (\kappa \widetilde{f} \xi^k)|^2_{\omega}e^{-k \varphi_0} \\
\lesssim & \sup_{B(p, 1/k)}(|\kappa \overline{\partial} \widetilde{f}|^2_{\omega}+ k^2|\kappa \widetilde{f} \overline{\partial} \xi|^2_{\omega})|\xi^k|^2 e^{-k\varphi_0}
\lesssim \sup_{B(p, 1/k)}O(|y|^{m}) e^{-k\varphi} 
= O(k^{-m})
\end{align*}
since $\varphi(x, y) = 2|y|^2 + O(|y|^3) \geq 0$ for $y$ with small norm.  
We also have 
\[
\int_{B(p, 1/k)} |b_k \xi^k|^2 e^{-k\varphi_0(p)} \frac{\omega^n}{n!} 
\lesssim \int_{B(p, 1/k)} |b_k \xi^k|^2 e^{-k\varphi_0} \frac{\omega^n}{n!} 
\leq \|\beta_k\|^2_{h^k} 
= O(k^{-m})
\]
for any $m \in \mathbb{N}$.  
Hence $b_k \to 0$ ($k \to \infty$) and $F_{k}/s^{k}$ converges to $f$ uniformly on $X$ as $k \to \infty$.  
\end{proof}
This $(F_k)_{k \in \mathbb{N}}$ will be used in the proof of Theorem~\ref{theorem:3} in Section~\ref{section:6}. 
\vspace{4mm}\\
\noindent $(\mathrm{B})$
Assume that $M$ is compact.  
Let $K_{k}(z, w)$ ($(z, w) \in M \times M$) be the Bergman kernel of $L^k$.  
$K_{k}(z, w)$ is a smooth kernel and we put 
$s_{f, k} = \left(\frac{\pi}{2k} \right)^{n/2}\int_{X} K_{k}(z, w) f(w)\zeta^k(w) dv_{X}(w)$.  
As stated in Section~\ref{section:1}, 
$\frac{s_{f, k}|_{X}}{\zeta^k}$ converges to $f$ uniformly on $X$ and that 
$
\lim_{k \to \infty} \left(\frac{2k}{\pi}\right)^{n/2} \|s_{f, k}\|^2_{h^k} =\int_{X} |f|^2 dv_{X}$ (see also Theorem~\ref{theorem:4}).  
Furthermore, $s_{f, k}$ has the following fundamental property.  
\begin{proposition}[Proposition~3.4 of \cite{Ioo}]\label{proposition:-1}
For any $g \in \Gamma(M, L^k)$, we have the following reproducing property: 
\[
(g, s_{f, k})_{h^k} = \left(\frac{\pi}{2k}\right)^{n/2}\int_{X} \langle g, f \zeta^k \rangle_{h^k} dv_{X}.  
\]
\end{proposition}
\begin{proof}
We have 
\begin{align*}
(g, s_{f, k})_{h^k} 
= &\left(\frac{\pi}{2k}\right)^{n/2} \int_{M} \langle g(z), \int_{X} K_k(z, w) f(w)\zeta^{k}(w) dv_X(w) \rangle_{h^k} \frac{\omega^n(z)}{n!} \\
= & \left(\frac{\pi}{2k}\right)^{n/2} \int_X \langle  \int_M K_k(w, z) g(z) \frac{\omega^n(z)}{n!}, f(w) \zeta^k(w)\rangle_{h^k} dv_{X}(w)\\ 
= & \left(\frac{\pi}{2k}\right)^{n/2} \int_{X} \langle g(w), f(w)\zeta^k(w) \rangle_{h^k} dv_{X}(w).   
\end{align*}
\end{proof}
We now give the proof of Proposition~\ref{proposition:0}.  
\begin{proof}[proof of Proposition~\ref{proposition:0}]
By Proposition~\ref{proposition:-1}, we have 
\[
\left(\frac{2k}{\pi}\right)^{n/2}(f_k, s_{f, k})_{h^k} = \int_{X} \langle f_{k}, f \zeta^k \rangle_{h^k} dv_{X} \to \int_{X} |f|^2 dv_X \quad (k \to \infty).  
\]
Hence 
\begin{align*}
0 \leq & \liminf_{k \to \infty}\left(\frac{2k}{\pi}\right)^{n/2} \|f_k -s_{f, k}\|^2_{h^k} = \liminf_{k \to \infty}\left(\frac{2k}{\pi}\right)^{n/2}\left(\|f_k\|^2_{h^k} -(s_{f, k}, f_k)_{h^k} - (f_k, s_{f, k})_{h^k} + \|s_{f, k}\|^2_{h^k} \right) \\
= & \liminf_{k \to \infty} \left(\frac{2k}{\pi} \right)^{n/2} \|f_k\|^2_{h^k} - \int_X |f|^2 dv_X.  
\end{align*}
\end{proof}

\section{Monge-Ampere measure on Grauert tubes}\label{section:3}
Let $(M, \omega)$ be a (non-compact) K\"ahler manifold of complex dimension $n$ with a complex structure $J$.   
Let $X \subset M$ be a totally real submanifold of real dimension $n$.  
Let $g_{\omega}(\cdot, \cdot) = \omega(\cdot, J\cdot)$ be the Riemannian metric on $M$.  
Let $g_X$ be the Riemannian metric on $X$ induced by $g_{\omega}$ and $dv_X$ be a volume density on $X$.  
We assume that there exists a non-negative smooth plurisubharmonic function $\rho$ on $M$ such that $\omega = \frac{i}{2}\partial \overline{\partial} \rho$, $\rho^{-1}(0) = X$, and that $\sqrt{\rho}$ satisfies the following Monge-Amp\`ere equation on $M \setminus X$:  
$(i\partial \overline{\partial} \sqrt{\rho})^n = 0$.  
In this case, we call $M$ a Grauert tube, and  
$X$ naturally forms a Lagrangian submanifold of $(M, \omega)$.  
We have $\sqrt{\rho}$ is a continuous plurisubharmonic function on $M$, and 
$(dd^c \sqrt{\rho})^n = (2i\partial \overline{\partial} \sqrt{\rho})^n$ defines a Monge-Amp\`ere measure whose support is contained in $X$.  
See Chapter~III, Section~3 of \cite{Dem} for the definition of the Monge-Amp\`ere measure.  
In this section, we are only concerned with a small neighborhood of the submanifold $X$.   
Hence we may regard $M$ as a tubular neighborhood of $X$.  
By identifying the normal bundle $JTX$ with $TX$, 
we consider $M$ as an open subset of $TX$.  
Then it follows that $\rho(\nu) = 2|\nu|^2$ for $\nu \in TX$ (see e.g., Theorem~3.1 of \cite{Lem-Szo} and the references therein).  
Here $|\nu|$ is the norm of $\nu$ with respect to the Riemannian metric $g_{X}$ on $X$ induced by $g_{\omega}$.  
\begin{proposition}\label{proposition:1}
Assume that $M$ is a Grauert tube.  
As a measure on $X$, we have 
\[
(dd^c \sqrt{\rho})^n = \sqrt{2}^{n} n!\sigma_n dv_{X}, 
\]
where $\sigma_n = \frac{\pi^{n/2}}{\Gamma(\frac{n}{2} + 1)}$ is the volume of an $n$-dimensional unit ball.  
\end{proposition}
For the proof, we need the following Demailly-Jensen-Lelong formula:  
\begin{theorem}[(6.5) of \cite{Dem}]\label{theorem:DJL}
Let $M$ be a Stein manifold and $\phi$ be a continuous plurisubharmonic function on $M$.  
Assume that the sublevel set $B_{\phi}(r) = \{z \in M\, |\,  \phi(z) < r\} \subset M$ is relatively compact for any $r < \sup_{M}\phi$.  
Then 
\[
\int \widetilde{u} d\mu_{r} - \int_{B_{\phi}(r)} \widetilde{u} (dd^c \phi)^n = \int_{-\infty}^{r} dt \int_{B_{\phi}(t)}  dd^c \widetilde{u} \wedge (dd^c \phi)^{n-1} 
\]
for any $\widetilde{u} \in C^{\infty}(M)$.  
Here $d\mu_r$ is a measure whose support is contained in $\{z \in M\, |\, \phi(z) = r\}$.  
If $\phi$ is smooth and $d\phi \neq 0$ on a neighborhood of $\partial B_{\phi}(r)$, 
$d\mu_r$  is equal to the pullback of $d^c \phi \wedge (dd^c\phi)^{n-1}$ by the inclusion map from $\partial B_{\phi}(r)$ to $M$.   
\end{theorem}
We note that $\widetilde{u}$ is assumed to be plurisubharmonic in \cite{Dem}.  
However, this assumption is not necessary since any smooth function can be represented as a difference of two plurisubharmonic functions.  
\begin{proof}[proof of Proposition~\ref{proposition:1}]
Let $u \in C^{\infty}(X)$ and let $\widetilde{u} \in C^{\infty}(M)$ be an almost holomorphic extension of $u$, that is, $\widetilde{u} = u$ and $\overline{\partial}\widetilde{u}$ vanishes to infinite order on $X$.   
The Demailly-Jensen-Lelong formula (Theorem~\ref{theorem:DJL}) gives 
\begin{align}\label{equation:1}
\int_{\partial B_{\sqrt{\rho}}(r)} \widetilde{u} \iota^{*}(d^c\sqrt{\rho} \wedge (dd^c \sqrt{\rho})^{n-1}) - \int_{X} u (dd^c \sqrt{\rho})^n= \int_{0}^{r} dt \int_{B_{\sqrt{\rho}}(t)}  dd^c \widetilde{u} \wedge (dd^c \sqrt{\rho})^{n-1}
\end{align}
for almost all $r$, where $B_{\sqrt{\rho}}(r)$ is a sublevel set of $\sqrt{\rho}$ and $\iota: \partial B_{\sqrt{\rho}}(r) \to M$ is the inclusion map .  
Since $dd^c \widetilde{u}(\nu) = O(|\nu|^{m})$ for any $m \in \mathbb{N}$ and $\nu \in TX \cap M$, we have 
\begin{align*}
& \int_0^{r} dt \int_{B_{\sqrt{\rho}}(t)}  dd^c \widetilde{u} \wedge (dd^c \sqrt{\rho})^{n-1} \\
= & \frac{1}{2^{n-1}}\int_0^{r}dt \int_{\|\nu\| \leq t} \frac{1}{\rho^{(n-1)/2}}dd^{c} \widetilde{u} \wedge (dd^c \rho)^{n-2}\wedge\left(dd^c \rho -\frac{n}{2\rho} d\rho \wedge d^c \rho \right) 
 = O(r^{m-n}).  
\end{align*}
Multiplying (\ref{equation:1}) by $r^{n-1}$ 
and integrating with respect to $r$ from $0$ to $\tau>0$, we have 
\begin{align}\label{equation:11}
\int_{B_{\sqrt{\rho}}(\tau)} \widetilde{u} \rho^{(n-1)/2} d\sqrt{\rho} \wedge d^c \sqrt{\rho} \wedge (dd^c \sqrt{\rho})^{n-1}
= \frac{\tau^n}{n} \int_{X} u (dd^c \sqrt{\rho})^{n} + O(\tau^{m}). 
\end{align}
Since $(dd^{c} \sqrt{\rho})^n = 0$ on $M \setminus X$, we have 
$\frac{n}{2\rho} d\rho \wedge d^{c}\rho \wedge (dd^c \rho)^{n-1} = (dd^c \rho)^n = 4^n \omega^n$. 
Then the left hand side of (\ref{equation:11}) is equal to 
\begin{align*}
& \frac{1}{2^{n+1}}\int_{B_{\sqrt{\rho}}(\tau)} \widetilde{u} \frac{1}{\rho} d\rho \wedge d^c \rho \wedge (dd^c \rho)^{n-1}
= \frac{1}{2^{n}n} \int_{B_{\sqrt{\rho}}(\tau)} \widetilde{u} (dd^{c}\rho)^n = 2^n (n-1)! \int_{B_{\sqrt{\rho}}(\tau)} \widetilde{u}\frac{\omega^n}{n!}.  
\end{align*} 
Let $\Omega \subset X$ be an open subset of $X$ such that there exist pointwise orthonormal vector fields $E_1, \ldots, E_n$ on $\Omega$.  
Let $\Pi:TX \to X$ be the projection map.  
We take a smooth local coordinate system $(x, y)$ on $\Pi^{-1}\Omega$ such that $(x, 0)$ corresponds to a point in $\Omega \subset \Pi^{-1}\Omega$, and $(x, y)$ corresponds to $y_1 (E_{1})_{(x, 0)} + \cdots + y_n (E_{n})_{(x, 0)} \in \Pi^{-1}\Omega$.  
Let $dv_Xdy_1 \cdots dy_n$ be a product measure on $\Pi^{-1}\Omega \simeq (X \cap V) \times \mathbb{R}^n$.  
Then $\frac{\omega^n}{n!} = (1 + O(|y|))dv_Xdy_1 \cdots dy_n$ around $\Omega \subset \Pi^{-1}\Omega$.  
Since $\rho(x, y) = 2|y|^2$, we have $B_{\sqrt{\rho}}(\tau) \cap \Pi^{-1}\Omega = \{(x, y) |\, |y| < \tau/\sqrt{2}\}$.    
Hence we obtain 
\[
\int_{B_{\sqrt{\rho}}(\tau)} \widetilde{u}\frac{\omega^n}{n!} = \frac{\tau^n\sigma_n}{2^{n/2}}\int_{X}  udv_X
+o(\tau^{n})
\] 
Taking $\tau \to 0$ completes the proof.  
\end{proof}

\begin{proposition}\label{proposition:2}
Assume that $M$ is a Grauert tube.  
Let $(r_k)_{k \in \mathbb{N}}$ be a sequence of small positive numbers which satisfies $kr_k^2 \to \infty$ as $k \to \infty$.   
Let $f_{k}$ ($k \in \mathbb{N}$) be a holomorphic function on $B_{\sqrt{\rho}}(r_k) = \{z \in M\,|\, \sqrt{\rho}(z) < r_k\}$.   
Then 
\[
\liminf_{k \to \infty} \left(\frac{2ck}{\pi}\right)^{n/2} \int_{B_{\sqrt{\rho}}(r_k)} |f_k|^2 e^{-ck\rho} \frac{\omega^n}{n!} \geq  \liminf_{k \to \infty} \int_{X} |f_k|^2 dv_{X}
\]
for any $c > 0$.  
\end{proposition}
\begin{proof}
By Demailly-Jensen-Lelong formula (Theorem~\ref{theorem:DJL}), 
we have 
\[
\int_{\partial B_{\sqrt{\rho}}(r)} |f_k|^2 \iota^{*}(d^{c} \sqrt{\rho} \wedge (dd^c \sqrt{\rho})^{n-1}) \geq \int_{X} |f_k|^2 (dd^c \sqrt{\rho})^n  
\]
for almost all $r$.  
Multiplying by $r^{n-1}e^{-ckr^2}$ and integrating with respect to $r$ from $0$ to $r_k$, we obtain 
\begin{align}\label{equation:2}
\int_{B_{\sqrt{\rho}}(r_k)} |f_k|^2 \rho^{(n-1)/2}e^{-ck\rho}  d\sqrt{\rho} \wedge d^c \sqrt{\rho} \wedge (dd^c \sqrt{\rho})^{n-1} \geq   
\int_0^{r_k} r^{n-1}e^{-ckr^2}dr \int_{X} |f_k|^2 (dd^c \sqrt{\rho})^n 
\end{align}
Then 
\[
\int_0^{r_k} r^{n-1}e^{-ckr^2}dr = \frac{1}{2(ck)^{n/2}}\left(\Gamma \left(\frac{n}{2}\right) - \int_{ckr_{k}^2}^{\infty} t^{\frac{n}{2}-1}e^{-t}dt\right), 
\]
and Proposition~\ref{proposition:1} gives that the $\liminf_{k \to \infty}$ of the right hand side of (\ref{equation:2}) multiplied by $(ck)^{n/2}$ is equal to  
$\liminf_{k \to \infty} 2^{n/2}(n-1)! \pi^{n/2} \int_{X} |f_k|^2 dv_{X}$.  
As in the proof of Proposition~\ref{proposition:1}, the left hand side of (\ref{equation:2}) is equal to 
$2^n(n-1)!\int_{B_{\sqrt{\rho}}(r_k)} |f_k|^2e^{-ck\rho} \frac{\omega^n}{n!}$ and this completes the proof.  
\end{proof}

\section{Micro-local estimate of holomorphic sections}\label{section:4}
We now return to the case where $M$ is not necessarily a Grauert tube.  
Let $(M, \omega)$ be a K\"ahler manifold of complex dimension $n$.   
Let $(L, \nabla, h)$, $X$, $s \in C^{\infty}(U, L)$ be as in Section~\ref{section:2}.   
Let $d_{X}(z)$ be the distance from $z \in M$ to $X$ with respect $g_{\omega}$. 
Our first main theorem is the following:  
\begin{theorem}\label{theorem:2} 
Let $(r_k)_{k \in \mathbb{N}}$ be a sequence of small positive numbers such that 
$
\lim_{k \to \infty} kr_k^2 = \infty.  
$ 
Let $V_k = \{z \in M\, |\, d_{X}(z) < r_k\}$, and let $f_k \in \Gamma(V_k, L^k)$.  
Then 
\[
\liminf_{k \to \infty} \left(\frac{2k}{\pi}\right)^{n/2} \|f_k\|^2_{V_{k}, h^k} \geq \liminf_{k \to \infty}\int_{X} |f_k|_{h^k}^2 dv_X.  
\]
If $\frac{f_k}{s^{k}}$ converges to $f \in L^2(X, dv_X)$ on $X$ in the $L^2$-norm, we have 
\[
\liminf_{k \to \infty} \left(\frac{2k}{\pi}\right)^{n/2} \|f_k\|^2_{V_{k}, h^k} \geq \int_{X} |f|^2 dv_X.  
\]
\end{theorem}
As the proof of Theorem~\ref{theorem:2} is very similar to that of Theorem~1 of \cite{Tib}, 
we only provide references for the proofs of the lemmas in what follows.  
By replacing $r_k$ with $\min\{r_k, \frac{\log k}{\sqrt{k}}\}$, we may assume that $(r_k)_{k \in \mathbb{N}}$ satisfies $\lim_{k \to \infty} kr_k^2 = \infty$ and $\lim_{k \to \infty} kr_k^3 = 0$.  
Whitney's theorem (Theorem~1 of \cite{Whi2}) gives that there exists a real analytic manifold $Y$ which is diffeomorphic to $X$.  
By a theorem of Bruhat and Whitney (\cite{Whi-Bru}, \cite{Gra}), there exists a Stein manifold $N$ of complex dimension $n$, which contains $Y$ as a real analytic and totally real submanifold.   
Taking almost holomorphic extension of a diffeomorphism from $Y$ to $X$, 
there exists a diffeomorphism $\Phi$ from a neighborhood of $Y$ to a neighborhood of $X$ such that $\Phi(Y) = X$ and $\overline{\partial} \Phi$ vanishes to any order on $Y$ (cf. Proposition~5.55 of \cite{Cie-Eli}).  
We may assume that $\Phi$ is a diffeomorphism from $N$ to $M$ by shrinking $N$ and $M$.  
Let $\varphi = -\log |s|^2_{h^k}$ and let 
$\omega' = \frac{dd^c \varphi \circ \Phi}{4}$.  
Let $d_{Y}(z)$ be the distance from $z \in N$ to $Y$ with respect to the Riemannian metric induced by $\omega'$.  
There exists $C > 1$ such that $\frac{1}{C}d_{X}(\Phi(z)) \leq d_{Y}(z) \leq C d_{X}(\Phi(z))$.  
Put $W_k = \{z \in N\, |\, d_{Y}(z) < \frac{r_k}{2C}\}$.  
Let $u_k \in C^{\infty}(V_k)$ such that $f_k = u_k s^k$. 
Then 
\begin{lemma}\label{lemma:1}
For any $m \in \mathbb{N}$, we have 
\[
\int_{W_k} |\overline{\partial}(u_k \circ \Phi)|^2_{\omega'} e^{-k\varphi \circ \Phi} \frac{(\omega')^n}{n!} = O(r_{k}^m) \|f_{k}\|^2_{V_{k}, h^k}.  
\]
\end{lemma}
For the proof we refer to the proof of Lemma~1 of \cite{Tib}.  
In Lemma~1 of \cite{Tib}, there exists the additional assumption such that $\|f_{k}\|^2_{h^k} = 1$.  
However this assumption has very little effect on the proof.  
Therefore, the proof can be carried out in exactly the same manner as in Lemma~1 of \cite{Tib}.  

By regarding $u_k \circ \Phi$ as a section of $T^{(n, 0)} N|_{W_k}$-valued $(n, 0)$-form, 
H\"{o}rmander $L^2$-estimate (cf. \!\!\!Chapter~VIII, Theorem~6.1 of \cite{Dem}, Corollary~5.3 of \cite{Dem2}) shows that there exists $\tilde{u}_k \in C^{\infty}(W_k)$ such that $\overline{\partial} \tilde{u}_k = \overline{\partial} (u_k \circ \Phi)$ and that 
\begin{align}\label{equation:a}
\int_{W_k} |\tilde{u}_k|^2 e^{-k \varphi \circ \Phi}\frac{(\omega')^n}{n!} = O(r_{k}^m) \|f_{k}\|^2_{V_k}
\end{align} 
for any $m \in \mathbb{N}$.  
Put $v_k = u_k \circ \Phi - \tilde{u}_k$.  
Then $v_k$ is a holomorphic function on $W_k$.   
%$\int_{W_k} |g_k|^2 e^{-k\varphi\circ\Phi} \frac{(\omega')^n}{n!} = (1 + O(r_{k}^m))\|f_{k}\|^2_{V_k}$.  
\begin{lemma}\label{lemma:2}
For any $m \in \mathbb{N}$, we have 
\[
|\tilde{u}_k|^2 = O(r_{k}^m) \|f_k\|^2_{V_k, h^k}
\]
uniformly on Y.  
\end{lemma}
For the proof we refer to the proof of Lemma~2 and Lemma~3 of \cite{Tib}.  

Let $\varepsilon > 0$ be any positive number.  
There exists a real analytic function $\varphi_{\epsilon}$ on a neighborhood of $Y$ which satisfies 
$|\varphi_{\varepsilon} - \varphi \circ \Phi|_{C^2} < \varepsilon$ (see Lemma~5 of \cite{Whi}).  
By a theorem of Guillemin and Stenzel~\cite{Gui-Ste}, 
there exists a real-analytic strictly plurisubharmonic function $\rho$ on a neighborhood of $Y$ which satisfies the following three conditions: 
$(i)$ $0 \leq \rho \leq 1$ and $\rho^{-1}(0) = Y$.  
$(ii)$ The Riemannian metric on Y induced by $\frac{dd^c \rho}{4}$ is equal to that induced by $\frac{dd^c \varphi_{\varepsilon}}{4}$.  
$(iii)$ $(dd^c \sqrt{\rho})^n = 0$ outside $Y$.  
We may assume $\varphi_{\varepsilon}$ and $\rho$ are defined on $N$ by shrinking $N$ if necessary.  
Hence $N$ is a Grauert tube.  
Put $\omega_{\rho} = \frac{dd^c \rho}{4}$.  
Let $dv_{Y, \rho}$, $dv_{Y, \varphi \circ \Phi}$ be Riemannian densities on $Y$ induced by $\omega_{\rho}, \omega'$ respectively.  
By shrinking $N$ further if necessary, there exists a constant $c > 0$ which does not depend on $\varepsilon$ and satisfies 
\[
(1-c\varepsilon) \varphi \circ \Phi \leq \rho \leq (1+c \varepsilon) \varphi \circ \Phi, \quad 
(1-c\varepsilon) (\omega')^n \leq (\omega_{\rho})^n \leq (1 + c \varepsilon) (\omega')^n, 
\]
\[
(1-c\varepsilon) dv_{Y, \varphi \circ \Phi} \leq dv_{Y, \rho} \leq (1+c\varepsilon) dv_{Y, \varphi \circ \Phi}
\]  
on $N$ 
(see the first part of Section~4 of \cite{Tib}).  
By Proposition~3, we have 
\begin{align*}
& \liminf_{k \to \infty} (1+c\varepsilon)\left(\frac{2k}{(1-c\varepsilon)\pi}\right)^{n/2} \int_{W_{k}} |v_{k}|^2 e^{-k \varphi \circ \Phi} \frac{(\omega')^n}{n!} \\
\geq &  
\liminf_{k \to \infty} \left(\frac{2k}{(1-c\varepsilon)\pi}\right)^{n/2} \int_{W_{k}} |v_{k}|^2 e^{-k \frac{\rho}{1-c\varepsilon}} \frac{(\omega_{\rho})^n}{n!} \\
\geq & \liminf_{k \to \infty} \int_{Y} |v_{k}|^2 dv_{Y, \rho} \\
\geq & 
\liminf_{k \to \infty} (1-c\varepsilon) \int_{Y} |v_k|^2 dv_{Y, \varphi \circ \Phi}.  
\end{align*}
Taking $\varepsilon \to 0$, we have 
\[
\liminf_{k \to \infty} \left(\frac{2k}{\pi} \right)^{n/2} \int_{W_k} |v_{k}|^2 e^{-k\varphi \circ \Phi} \frac{(\omega')^n}{n!} \geq \liminf_{k \to \infty} \int_{Y} |v_k|^2 dv_{Y, \varphi \circ \Phi}.  
\]
Since $\Phi^{-1}$ is almost holomorphic on $X$, we have 
$dv_{X} = (\Phi^{-1}|_{X})^{*}dv_{Y, \varphi \circ \Phi}$ on $X$ and $(\Phi^{-1})^{*} \omega' = (1+ O(r_{k}^m)) \omega$ on $\Phi(W_k)$ for any $m \in \mathbb{N}$.   
For any $\varepsilon > 0$, 
(\ref{equation:a}) and Lemma~\ref{lemma:2} show 
\begin{align*}
& \|f_{k}\|_{V_k, h^k}^2 \geq \int_{\Phi(W_k)} |u_{k}|^2 e^{-k\varphi} \frac{\omega^n}{n!}
\geq (1-\varepsilon) \int_{W_k}|u_k \circ \Phi|^2 e^{-k\varphi \circ \Phi} \frac{(\omega')^n}{n!}\\ 
\geq & (1-\varepsilon)\left( \frac{1}{1+\varepsilon}\int_{W_k} |v_{k}|^2 e^{-k \varphi \circ \Phi} \frac{(\omega')^n}{n!} 
- \frac{O(r_{k}^m)}{\varepsilon}\|f_{k}\|_{V_k, h^k}^2\right) 
\end{align*}
for sufficiently large $k$.  
Here we used $|c+d|^2 \geq \frac{1}{1 + \varepsilon} |c|^2 - \frac{1}{\varepsilon} |d|^2$ for $c, d \in \mathbb{C}$.  
We also have 
\begin{align*}
\int_X |f_k|^2_{h^k} dv_X = \int_{X} |u_k|^2 dv_{X} \leq (1+\varepsilon)\int_{Y} |v_{k}|^2 dv_{Y, \varphi \circ \Phi}  + O(r_{k}^m) \left(1+\frac{1}{\varepsilon}\right)\|f_{k}\|^2_{V_k, h^k}.  
\end{align*}
Hence 
\[
\liminf_{k \to \infty} \left(\frac{2k}{\pi}\right)^{n/2} \|f_{k}\|^2_{V_k, h^k} \geq \liminf_{k \to \infty}\frac{1-\varepsilon}{(1+\varepsilon)^2} \int_{X}|f_k|^2_{h^k}dv_X.  
%\liminf_{k \to \infty} \left(\frac{2k}{\pi}\right)^{n/2} \|f_{k}\|^2 \geq \liminf_{k \to \infty} \frac{1}{1+k^{n/2}O(r_k^m)}\left(\int_{X} |f_{k}|^2 dv_{X} + O(k^{n/2}r_k^{m})\|f_{k}\|^2 \right)
\]
By taking $\varepsilon \to 0$, we complete the proof of Theorem~\ref{theorem:2}.  

\begin{remark}
With almost no change to the proof of Theorem~\ref{theorem:2}, the following slight generalization holds: \\
Let $(j_k)_{k \in \mathbb{N}}$ be a strictly increasing sequence of positive integers.  
Let $(r_{k})_{k \in \mathbb{N}}$ be a sequence of small positive numbers such that 
$
\lim_{k \to \infty} j_k r_{k}^2 = \infty.  
$ 
Let $V_k = \{z \in M\, |\, d_{X}(z) < r_k\}$, and let $f_k \in \Gamma(V_k, L^{j_k})$.  
Then 
\[
\liminf_{k \to \infty} \left(\frac{2j_k}{\pi}\right)^{n/2} \|f_k\|^2_{V_{k}, h^{j_k}} \geq \liminf_{k \to \infty}\int_{X} |f_k|_{h^{j_k}}^2 dv_X.  
\]
\end{remark}

\section{Horizontal and Vertical subspaces of tangent spaces}\label{section:5}
In this section, we review some definitions and properties on the geometry of tangent bundles.   
Let $(X, g)$ be a Riemannian manifold of real dimension $n$, $\nabla^{LC}$ be the Levi-Civita connection of $g$, and $\Pi : TX \to X$ be the projection map.  
Let $\nu \in TX$ and $\Pi(\nu) = p \in X$.  
We denote the kernel of $d\Pi_{\nu}: T_{\nu}TX \to T_{p}X$ at $\nu$ by 
$\mathcal{V}_{\nu}$ and call it the vertical subspace of $T_{\nu}TX$.  
Let $E_{p} \in T_{p}X$.  
As there exists a natural isomorphism from $\mathcal{V}_{\nu}$ to $T_{p}X$, 
we can define the vertical lift $E^v_{p, \nu} \in \mathcal{V}_{\nu}$ of $E_p$ at $\nu$.  
Let $\gamma(t) = \mathrm{exp}_{p}(tE_p)$ be the exponential map induced by $\nabla^{LC}$.  
Then we have the parallel transport $\nu(t) \in TX$ of $\nu$ along $\gamma(t)$.  
That is,  $\nu(t)$ satisfies $\nu(0) = \nu$, $\Pi (\nu(t)) = \gamma(t)$ and $\nabla^{LC}_{\gamma'(t)} \nu(t) = 0$.  
The horizontal lift $E^{h}_{p, \nu}$ of $E_p$ is defined by $E_{p, \nu}^{h} = \nu'(0) \in T_{\nu}TX$.  
The set of all horizontal lift $\mathcal{H}_{\nu} = \{E_{p, \nu}^{h} \in T_{\nu}TX\, |\, E_p \in T_{p}X\}$ defines a $n$-dimensional subspace of $T_{\nu}TX$ and $\mathcal{H}_{\nu}$ is called the horizontal subspace of $T_{\nu}TX$.  
Then $T_{\nu}TX = \mathcal{H}_{\nu} \oplus \mathcal{V}_{\nu}$.  
If $E$ is a vector field on an open subset $\Omega$ of $X$, we have the vertical lift $E^v$ and horizontal lift $E^h$ of $E$.  
$E^v$ and $E^h$ are vector fileds on $\Pi^{-1}(\Omega)$.  
Let $\{E_1, \ldots, E_n\}$ be pointwise orthonormal vector fields on an open subset $\Omega \subset X$.  
We denote the dual basis of $\{E_1^h, \ldots, E_n^h, E_1^v, \ldots, E_n^v\}$ by $\{E_h^1, \ldots, E_h^n,  E_v^1, \ldots, E_v^n\}$, that is, $E_{h}^j$ and $E_{v}^j$ are one-forms on $\Pi^{-1}\Omega$ such that $\langle E_{i}^h, E_{h}^j\rangle = \delta_{ij}$, $\langle E_{i}^v, E_{v}^j\rangle = \delta_{ij}$, $\langle E_{i}^h, E_{v}^j\rangle = 0$.  
%Then $\hat{g} = \sum_{j = 1}^n (E_{h}^j)^2 + \sum_{j=1}^n (E_v^j)^2$.  
By identifying $\nu \in TX$ with $g(\cdot, \nu) \in T^{*}X$, 
we identify $TX$ with $T^*X$.  
There exists the canonical symplectic form $\omega_0$ on $TX$.  
It is known that $\omega_0 = \sum_{j=1}E_{v}^j \wedge E_{h}^j$ on $\Pi^{-1}\Omega$ (see \cite{Kli}).  
Let $\Psi_{k}: TX \to TX$ be multiplication by $\frac{1}{\sqrt{k}}$ in the fibers of $TX$.  
We have $\Psi_{k *}E_{j}^h = E_{j}^h$, $\Psi_{k *}E_j^v = \frac{1}{\sqrt{k}} E_{j}^v$ and $\Psi_{k}^* \omega_0 = \frac{1}{\sqrt{k}}\omega_0$.  

Now we assume that there exists a complex structure $J$ on a neighborhood $U \subset TX$ of $X$ 
such that $E_{j}^v = JE_{j}^h$ ($j = 1, \ldots, n$) on $X$.  
We also assume that there exists a K\"ahler form $\omega$ on $U$ such that $X$ is a Lagrangian submanifold of $(U, \omega)$, and that the Riemannian metric on $X$ induced by $\omega$ is equal to $g$.  
Then $\omega(E_{i}^h, E_{j}^v) = -\omega(E_{i}^v, E_{j}^h) = g(E_i, E_j) =\delta_{ij}$ and $\omega(E_{i}^h, E_{j}^h) = \omega(E_i^v, E_j^v) = 0$ on $X$.  
Hence $\omega|_{X} = -\omega_0|_{X}$, 
and $\lim_{k \to \infty}(\sqrt{k}\Psi_{k}^*\omega + \omega_0) = 0$ uniformly on any compact set of $TX$.  
For any $\nu \in TX$, we have 
$JE_{j}^v (\nu) = -E_{j}^{h}(\nu) + \sum_{i=1}^n O(|\nu|) E_{i}^v(\nu) + \sum_{i=1}^n O(|\nu|) E_{i}^h (\nu)$.  
Then we have 
\begin{align*}
&\lim_{k \to \infty} \frac{1}{\sqrt{k}}\left((\Psi_{k *})^{-1} (E_{j}^v+\sqrt{-1}JE_{j}^v)\right) (\nu) \\
= & \lim_{k \to \infty}\frac{1}{\sqrt{k}} \left(\sqrt{k}E_{j}^v(\nu) - \sqrt{-1}E_{j}^h(\nu) + \sum_{i=1}^nO(\frac{|\nu|}{\sqrt{k}}) \sqrt{k}E_{i}^v(\nu) + \sum_{i=1}^n O(\frac{|\nu|}{\sqrt{k}})E_{i}^h(\nu)\right) 
= E_{j}^v(\nu)
\end{align*}
uniformly for $\nu \in TX$ with bounded norm.  
This implies that the pullback of K\"ahler polarization by $\Psi_k$ converges to the vertical polarization as $k \to \infty$.  

\section{Convergence of holomorphic sections on the tangent bundle}\label{section:6}
Let $(M, \omega)$ be a K\"ahler manifold of complex dimension $n$.   
Let $(L, \nabla, h)$, $X$, $U$, $s \in C^{\infty}(U, L)$ be as in Section~\ref{section:2}.   
Let $(r_k)_{k \in \mathbb{N}}$ be a sequence of small positive numbers.  
Let $d_{X}(z)$ be the distance from $z \in M$ to $X$ with respect to the Riemannian metric $g_{\omega}(\cdot, \cdot) = \omega(\cdot, J\cdot)$.  
Let $V_k = \{z \in M\, |\, d_{X}(z) < r_k\}$.  
By identifying $\nu \in TX$ with $J\nu$, we identify $TX$ with the normal bundle $JTX \subset TM|_{X}$ of $X$.  
Hence we consider $V_k$ and $U$ as open subsets of $TX$ as in Section~\ref{section:3}. 
Let 
$f_k \in \Gamma(V_k, L^k)$.  
Put $f_k = u_k s^k$ ($u_k \in C^{\infty}(V_k)$).  
We say that $(f_k)_{k \in \mathbb{N}}$ satisfies $(\star)$ if $(f_k)_{k \in \mathbb{N}}$ satisfies the following two conditions $(\star 1), (\star 2)$: 
\begin{quote}
($\star 1$)
$\|f_{k}\|^2_{V_k, h^k} = O(k^{-n/2})$.  \\
($\star 2$)
For any smooth vector field $E$ on $X$ and any $\varepsilon > 0$, there exists $k_{0} \in \mathbb{N}$ such that 
\[
\int_{V_k} |E^h u_{k}|^2 e^{-k\varphi}\frac{\omega^n}{n!} \leq \varepsilon k \|f_{k}\|^2_{V_{k}, h^k}  
\]
for $k > k_0$.  
Here $E^h$ is the horizontal lift of $E$ (see Section~\ref{section:5}).
\end{quote} 
The following is our second main theorem:  
\begin{theorem}\label{theorem:3}
Let $(r_k)_{k \in \mathbb{N}}$ be a sequence of positive numbers such that 
\[
\lim_{k \to \infty} kr_k^2 = \infty \quad \text{and} \quad \lim_{k \to \infty} kr_k^3 = 0. 
\] 
Let 
$V_k = \{z \in M\, |\, d_{X}(z) < r_k\}$ and $V'_k = \{z \in M\, |\, d_{X}(z) < r_k/2\}$.  
Let $f_k \in \Gamma(V_k, L^k)$.  
Assume that $(f_{k})_{k \in \mathbb{N}}$ satisfies $(\star)$.  
Then there exists a sequence of smooth functions $(a_{k})_{k \in \mathbb{N}}$ on $X$ such that 
$\lim_{k \to \infty}k^{n/2}\| f_k - (\Pi^*a_k) s^k\|^2_{V'_k, h^k} = 0$.  
Here we consider $V_{k} \subset TX$, and $\Pi: TX \to X$ is the projection map.  
\end{theorem}

Let us begin the proof of Theorem~\ref{theorem:3}.  
First we assume that $\liminf_{k \to \infty}\int_X |f_k(x)|_{h^k}^2 dv_X > 0$.  
In this case, we have the following.  
\begin{lemma}\label{lemma:3}
Under the same assumption as in Theorem~\ref{theorem:3}, we further assume that 
$\liminf_{k \to \infty} \int_X |f_k(x)|^2_{h^k} dv_X > 0$.  
Then $\|f_{k}\|^2_{V_{k}, h^k} = O(\|f_{k}\|^2_{V'_{k}, h^k})$.  
\end{lemma}
\begin{proof}
Let $C > 0$ such that $\|f_{k}\|^2_{V_{k}, h^k} \leq C k^{-n/2}$.  
By Theorem~\ref{theorem:2}, we have 
\[
\left(\frac{2k}{\pi}\right)^{n/2}\|f_{k}\|^2_{V'_{k}, h^k} \geq \frac{1}{2}\liminf_{k \to \infty}\int_{X}|f_k(x)|_{h^k}^2 dv_X > 0
\]
for large $k$.  
Hence $\|f_{k}\|^2_{V_k, h^k} \leq C\left(\frac{2}{\pi}\right)^{n/2} \left(\frac{1}{2}\liminf_{k \to \infty} \int_X |f_k(x)|_{h^k}^2 dv_x \right)^{-1} \|f_{k}\|^2_{V'_k, h^k}$ for large $k$.   
\end{proof}

Let $\Omega \subset X$ be an open subset of $X$ such that there exist pointwise orthonormal vector fields $E_1, \ldots, E_n$ on $\Omega$.  
We take a smooth local coordinate system $(x, y)$ on $\Pi^{-1}\Omega$ such that $(x, 0)$ corresponds to a point in $\Omega \subset \Pi^{-1}\Omega$, and $(x, y)$ corresponds to $y_1 (E_{1})_{(x, 0)} + \cdots + y_n (E_{n})_{(x, 0)} \in \Pi^{-1}\Omega$.  
This local coordinate system is essentially the same as the one introduced in Section~\ref{section:2}.  
We note that $|y| = \sqrt{y_1^2 + \cdots + y_n^2}$ is equal to the norm of tangent vector which corresponds to $(x, y)$, so one may assume that $|y|$ is a smooth function defined on the entire tangent space $TX$.  
Let $\varphi = -\log |s|^2$. 
Since $\frac{dd^c \varphi}{4} = \omega$, 
$\varphi (x, y) = 2|y|^2 + O(|y|^3)$ (see Section~\ref{section:2}).  
Let $E_{j}^h$ be the horizontal lift of $E_{j}$ and $E_{j}^v$ be the vertical lift of $E_{j}$ on $\Pi^{-1}\Omega$.  
We note that $E_{j}^v$ can be written as $\frac{\partial}{\partial y_j}$ in the local coordinate system $(x, y)$.  
Then $g_{\omega}(E_{i}^h, E_{j}^h) = \delta_{ij} + O(|y|)$, $g_{\omega}(E_{i}^v, E_{j}^v) = \delta_{ij} + O(|y|)$ and $g_{\omega}(E_{i}^h, E_{j}^v) = O(|y|)$.   
Let $J$ be a complex structure of $M$.   
Then $J E_{j}^h = E_{j}^v + O(|y|)$, 
$JE_{j}^v = -E_{j}^h + O(|y|)$.  
\begin{lemma}\label{lemma:4}
We have 
\[
\nabla_{E_{j}^h} s 
= 2\sqrt{-1} y_{j} s+ O(|y|^2), \quad 
\nabla_{E_{j}^v} s 
= -2y_j s + O(|y|^2)  
\]
around $\Omega \subset \Pi^{-1}\Omega$.  
\end{lemma}
\begin{proof}
Since $\overline{\partial}s$ vanishes to infinite order on $X$, 
the connection form of $(L, \nabla)$ associated to the trivialization $s$ can be written by $-\partial \varphi + O(|y|^{m})$ on $U$ for any $m \in \mathbb{N}$.  
By the definition of $E_{j}^h$ and $E_j^v$, we have 
$E_{j}^h |y|^2 = 0$, $E_{j}^v|y|^2 = 2y_{j}$.  
Hence we have 
$
\nabla_{E_{j}^h} s = - \left(\frac{1}{2}(E_{j}^h - \sqrt{-1} JE_{j}^h) \varphi\right) s
= 2\sqrt{-1} y_{j} s+ O(|y|^2).  
$
The second equation holds in the same way.  
\end{proof}
Let $\{\Omega_{\lambda}\}$ be a finite open covering of $X$ and $\{\eta_{\lambda}\}$ be a partition of unity subordinate to $\{\Omega_{\lambda}\}$.   
We regard $\eta_{\lambda}$ as a smooth function on $TX$ 
by extending $\eta_{\lambda}$ so that it is constant in the fiber direction.  
We may assume that there exists pointwise orthonormal vector fields $E_{\lambda, 1}, \ldots, E_{\lambda, n}$ on $\Omega_{\lambda}$.  
For the simplicity, we often omit the notation $\lambda$ and write $E^{h}_j = E^{h}_{\lambda, j}$, $E_{j}^v = E_{\lambda, j}^{v}$ if there is no fear of confusion.  
We take a smooth coordinate $(x_{\lambda, 1}, \ldots, x_{\lambda, n}, y_{\lambda, 1}, \ldots, y_{\lambda, n})$ on $T\Omega_{\lambda}$ as described above, and we also write $x_{\lambda, j}, y_{\lambda, j}$ as $x_{j}, y_{j}$.  
Let $\chi \in C^{\infty}_0(\mathbb{R})$ a smooth function such that $0 \leq \chi \leq 1$,  $\chi(t) = 1$ for $|t| \leq \frac{1}{2}$ and that $\chi(t) = 0$ for $|t| \geq 1$.  
Put $\chi_{k}(y) = \chi (r_k^{-1} |y|)$.  
We note that  $\chi_k$ is defined on the entire tangent bundle $TX$.  
\begin{lemma}\label{lemma:5}
Let $u_{k} \in C^{\infty}(V_{k})$ such that $f_k = u_k s^k$.  
Put $v_k = u_k e^{-k|y|^2}$.  
Assume the same assumptions as in Lemma~\ref{lemma:3}.  
For any small $\varepsilon > 0$, there exists $k_0 > 0$ such that 
\begin{align}\label{equation:6}
\sum_{\lambda} \sum_{j=1}^n \int \eta_{\lambda}|E_{j}^v (\chi_k v_k)|^2 \frac{\omega^n}{n!} + \int 4k^2|y|^2 |\chi_kv_k|^2\frac{\omega^n}{n!}  \leq (1+\varepsilon) 2kn \int |\chi_k v_k|^2 \frac{\omega^n}{n!}
\end{align}
for $k > k_0$.  
\end{lemma}
\begin{proof}
The Kodaira Laplacian $\Box=\overline{\partial}^{*} \overline{\partial}$ and the Bochner Laplacian $\Delta = \nabla^{*} \nabla$ of $(L^k, h^k)$ satisfy   
\[
\Delta g = 2\Box g+ 2kn g 
\]
for a smooth section $g$ of $L^k$ (see e.g., Section~2 of \cite{Hat}).  
The constant factor $2$ in the last term arises because $\sqrt{-1}F^{\nabla} = 2 \omega$.  
We have 
$
(\Delta f_k,  \chi^2_k f_k)_{h^k} = 2kn(f_k,  \chi^2_k f_k)_{h^k}, 
$ 
and the left hand side is equal to $(\nabla f_k, \nabla(\chi_k^2 f_k))_{h^k, \omega}$.  
Then 
\begin{align*}
&\sum_{\lambda} \sum_{j = 1}^n \left((\nabla_{E_{j}^h}f_{k}, \eta_{\lambda} \nabla_{E_{j}^h}(\chi^2_k f_k))_{h^k} + (\nabla_{E_{j}^v}f_k, \eta_{\lambda} \nabla_{E_{j}^{v}}( \chi^2_k f_k))_{h^k}\right)\\
 + & O(r_k) \sum_{\lambda, i, j, a, b} (\nabla_{E_{i}^{a}} f_k, \eta_{\lambda} \nabla_{E_{j}^{b}}(\chi^2_k f_{k}))_{h^k}
=  2kn \|\chi_k f_k\|^2_{h^k}, 
\end{align*}
where $a, b$ run through $\{v, h\}$.   
In the context below, the same letter $C>0$ will be used to denote different constants depending on the same set of arguments.
Let $\varepsilon > 0$ be any small number.  
The Cauchy-Schwarz inequality and $|E_{j}^{b} \chi_{k}| = O(r_{k}^{-1})$ imply  
\begin{align}
& |(\nabla_{E_{i}^a} f_{k}, \eta_{\lambda} \nabla_{E_{j}^b}( \chi^2_k f_{k}))_{h^k} - (\nabla_{E_{i}^a}f_{k},\eta_{\lambda} \chi^2_k \nabla_{E_{j}^b} f_k)_{h^k}| \label{equation:5} \\
= & |(\nabla_{E_{i}^a} f_k, 2 \eta_{\lambda} \chi_k f_k E_{j}^b \chi_k)_{h^k}| 
\leq
\varepsilon \| \sqrt{\eta_{\lambda}} \chi_k \nabla_{E_{i}^a} f_k\|^2_{h^k} + \frac{C}{\varepsilon r_{k}^{2}} \|\sqrt{\eta_{\lambda}} f_{k}\|^2_{V_k, h^k}.   \nonumber
\end{align}
We have $\|f_{k}\|^2_{V_{k}} \leq C \|\chi_k f_k\|^2$ by Lemma~\ref{lemma:3}.   
Hence, by rescaling $\varepsilon$ by a constant factor, we have  
\begin{align*}
\sum_{\lambda}\sum_{j = 1}^n \left(\|\sqrt{\eta_{\lambda}}\chi_k \nabla_{E_{j}^h}f_{k}\|^2_{h^k} + \|\sqrt{\eta_{\lambda}} \chi_k \nabla_{E_{j}^v}f_k\|^2_{h^k}\right) 
\leq  (1 + \varepsilon) 2kn \|\chi_k f_k\|^2_{h^k}.  
\end{align*}
for sufficiently large $k$.  
Here we used $\lim_{k \to \infty}\frac{1}{kr_{k}^2} = 0$.  
By Lemma~\ref{lemma:4}, 
\begin{align*}
& | \sqrt{\eta_{\lambda}} \chi_k \nabla_{E_j^h} f_k|^2_{h^k} 
=  \eta_{\lambda} \chi_k^2 |(E_j^h u_k + 2ky_j \sqrt{-1}u_k + kO(|y|^2)u_k )|^2 e^{-k\varphi} \\
\geq & \frac{1}{1 + \varepsilon} \eta_{\lambda} \chi_k^2 |2ky_j u_k|^2e^{-k\varphi} - 
\frac{2}{\varepsilon} \left(
\eta_{\lambda} \chi_k^2 |E_{j}^h u_k |^2 e^{-k\varphi} + k^2 r_k^{4}|u_k|^2 e^{-k\varphi}
\right) 
\end{align*}
and 
\begin{align*}
& |\sqrt{\eta_{\lambda}} \chi_k \nabla_{E_j^v}f_k|^2_{h^k} 
=  \eta_{\lambda} \chi_k^2 |E_j^v u_k -2ky_j u_k+ kO(|y|^2)u_k|^2 e^{-k\varphi} \\
\geq & \frac{1}{1 + \varepsilon} \eta_{\lambda} \chi_k^2 |E_{j}^v u_k - 2ky_j u_k|^2 e^{-k\varphi} - \frac{1}{\varepsilon} k^2 r_k^4 |u_k|^2 e^{-k\varphi}.  
\end{align*}
Here we used $|c+d|^2 \geq \frac{1}{1 + \varepsilon} |c|^2 - \frac{1}{\varepsilon} |d|^2$ for $c, d \in \mathbb{C}$.  
By condition $(\star 2)$, 
we have 
\[
\int \eta_{\lambda} \chi_k^2 |E_{j}^h u_k |^2 e^{-k\varphi} \frac{\omega^n}{n!} \leq \varepsilon^2 k 
\|\chi_k^2 f_k\|_{h^k}^2
\]
for sufficiently large $k$.  
Here we again used $\|f_k\|_{V_{k}}^2 \leq C \|\xi_k f_k\|^2$.  
Combining these and $\lim_{k \to \infty}kr_k^4 = 0$, 
we have 
\[
\int \left(  \sum_{\lambda}\sum_{j = 1}^n
\eta_{\lambda} \chi_k^2 |E_{j}^v u_k - 2y_j k u_k|^2 e^{-k\varphi}
+ 4k^2 \chi_k^2  |y|^2 |u_k|^2 e^{-k\varphi} 
\right) \frac{\omega^n}{n!} 
\leq (1 + \varepsilon) 2kn \|\chi_k f_k\|^2_{h^k}  
\]
for large $k$ if we rescale $\varepsilon$ by a constant factor.  
Put $v_k = u_k e^{-k|y|^2}$.  
Since $E_{j}^v|y|^2 = 2y_j$, we have 
\[
|E_{j}^v u_k - 2y_j k u_k|^2 e^{-k\varphi} = | E_{j}^v v_k|^2 e^{-kO(|y|^3)}.  
\]
The condition $\lim_{k \to \infty} kr_k^3 = 0$ gives $\lim_{k \to \infty}e^{-kO(|y|^3)} = 1$ uniformly on the support of $\chi_k$.  
Hence we have 
\begin{align*}
&\int \left(\sum_{\lambda} \sum_{j=1}^n \eta_{\lambda}\chi_k^2 |E_{j}^v v_k |^2 + 4k^2 \chi_k^2 |y|^2|v_k|^2
\right)\frac{\omega^n}{n!}  
\leq  (1 + \varepsilon) 2kn \int \chi_k^2 |v_k |^2 \frac{\omega^n}{n!} 
\end{align*}
for sufficiently large $k$.  
Here we again rescaled $\varepsilon$ by a constant factor.  
In a similar way to (\ref{equation:5}), the following inequality holds:  
\begin{align*}
&   |E_j^v(\chi_k v_k)|^2 - |\chi_k E_j^v v_k|^2   
= |v_k E_{j}^v \chi_k + \chi_k E_j^v v_k|^2  -  |\chi_k E_j^v v_k|^2  \\
\leq & \left(1 + \frac{1}{\varepsilon}\right) |v_k E_{j}^v \chi_k|^2 + \varepsilon |\chi_k E_{j}^v v_k|^2
\leq  \frac{Ck}{\varepsilon kr_k^2}  |v_k|^2 + \varepsilon |\chi_k E_j^v v_k|^2.  
\end{align*}
By Lemma~\ref{lemma:3}, we have $\int_{V_k} |v_k|^2 \frac{\omega^n}{n!} \leq C \int \chi_k^2 |v_k|^2 \frac{\omega^n}{n!}$ for sufficiently large $k$.  
Then, 
by rescaling $\varepsilon$ by a constant factor again, we complete the proof of Lemma~\ref{lemma:5}.  
\end{proof}

Let $dv_Xdy_{\lambda, 1} \cdots dy_{\lambda, n}$ be the product measure on $ \Pi^{-1}\Omega_{\lambda} \simeq \Omega_{\lambda} \times \mathbb{R}^n$.  
We have $\frac{\omega^n}{n!} = (1 + O(|y|))dv_Xdy_{\lambda, 1} \cdots dy_{\lambda, n}$ around $\Omega_{\lambda} \subset \Pi^{-1}\Omega_{\lambda}$.  
Here $|y| = |y_{\lambda}|$ does not depend on $\lambda$.  
By rescaling $\varepsilon$ by a constant factor, (\ref{equation:6}) shows
\begin{align*}
& \sum_{\lambda}\int \eta_{\lambda} \left(\sum_{j=1}^n |E_{j}^v (\chi_k v_k)|^2 +4 k^2 |y|^2 |\chi_k v_k|^2  \right)dv_X dy_{\lambda, 1}\cdots dy_{\lambda, n} \\
\leq & (1+\varepsilon)2kn \sum_{\lambda} \int \eta_{\lambda}\chi_{k}^2|v_k|^2 dv_Xdy_{\lambda, 1}\cdots dy_{\lambda, n}   
\end{align*}
for sufficiently large $k$.  
Let $\Psi_{k}: TX \to TX$ be multiplication by $\frac{1}{\sqrt{k}}$ in the fibers of $TX$.  
Put $\widetilde{v_k} = \Psi_k^*(\chi_k v_k)$.  
By 
pulling back with $\Psi_k$ and multiplying by $k^{n/2-1}$,  the above inequality shows 
\begin{align}
& \sum_{\lambda}\int \eta_{\lambda}\left(\sum_{j=1}^n  |E_{j}^v \widetilde{v_k}|^2 +4  |y|^2 |\widetilde{v_k}|^2 \right) dv_Xdy_{\lambda, 1}\cdots dy_{\lambda, n} \label{equation:7} \\
\leq & (1+\varepsilon)2n  \sum_{\lambda} \int \eta_{\lambda}|\widetilde{v_k}(y)|^2 dv_Xdy_{\lambda, 1}\cdots dy_{\lambda, n} \nonumber
\end{align}
for sufficiently large $k$.  
Now we denote the harmonic oscillator on $\mathbb{R}^n$ by $L = - \sum_{j=1}^n\frac{\partial^2}{\partial y_j^2} + 4 |y|^2$.  
Let $H_{p}(t) = e^{t^2/2}\frac{d^p}{dt^p} e^{-t^2}$ and 
let $H_{2, p}(t) = H_{p}(\sqrt{2}t)$.  
Then $L \prod_{j=1}^n H_{2, p_j}(y_j) = \sum_{j=1}^n 2(2p_j+1) \prod_{j=1}^n H_{2, p_j}(y_j)$, 
and $\{\prod_{j=1}^n H_{2, p_j}(y_j)\}_{p_1, \ldots, p_n \in \{0\} \cup \mathbb{N}}$ is an orthogonal basis of $L^2(\mathbb{R}^n)$.    
We consider $e^{-|y|^2}$ as a function on the entire tangent space $TX$.   
We take a decomposition $\widetilde{v_k} = \widetilde{v}_{k, 1} + \widetilde{v}_{k, 2}$, where 
\[
\widetilde{v}_{k, 1} = \Pi^{*} a_k e^{-|y|^2}, \quad  a_k(x_{\lambda}) = \left(\frac{2}{\pi}\right)^{n/2}  \int_{\mathbb{R}^n}\widetilde{v}_{k}(x_{\lambda}, y_{\lambda}) e^{-|y|^2} dy_{\lambda, 1} \cdots dy_{\lambda, n} 
\]
for $x_{\lambda} \in \Omega_{\lambda}$.  
Since $dy_{\lambda, 1} \cdots dy_{n, \lambda}$ is equal to $dy_{\lambda', 1} \cdots dy_{\lambda', n}$ as a measure on $T_{x}X$ for $x \in \Omega_{\lambda} \cap \Omega_{\lambda'}$, 
the definition of $a_k(x_{\lambda})$ does not depend on $\lambda$. 
Hence $a_k$ is a smooth function on $X$.  
We have 
$
\int \eta_{\lambda} \widetilde{v}_{k, 1} \overline{\widetilde{v}_{k, 2}} dv_Xdy_{\lambda, 1} \cdots dy_{\lambda, n} = 0.   
$
For any $x \in \Omega_{\lambda} \cap \Omega_{\lambda'}$, 
we have $\sum_{j=1}^n \frac{\partial^2}{\partial y_{\lambda, j}^2} = \sum_{j=1}^n \frac{\partial^2}{\partial y_{\lambda', j}^2}$ on $T_xX$, 
and we may consider that $L$ is a differential operator on $TX$. 
Then 
\[
\int \eta_{\lambda} L\widetilde{v}_{k, 1} \overline{\widetilde{v}_{k, 1}} dv_X dy_{\lambda, 1} \cdots dy_{\lambda, n} = 2n \int \eta_{\lambda} |\widetilde{v}_{k, 1}|^2 dv_X dy_{\lambda, 1} \cdots dy_{\lambda, n}, 
\]
\[
\int \eta_{\lambda} L\widetilde{v}_{k, 1} \overline{\widetilde{v}_{k, 2}} dv_X dy_{\lambda, 1} \cdots dy_{\lambda, n} = \int \eta_{\lambda} \widetilde{v}_{k, 1} L \overline{\widetilde{v}_{k, 2}} dv_X dy_{\lambda, 1} \cdots dy_{\lambda, n}= 0, 
\]
\[
\int \eta_{\lambda} L\widetilde{v}_{k, 2} \overline{\widetilde{v}_{k, 2}} dv_X dy_{\lambda, 1} \cdots dy_{\lambda, n} \geq (2n+4)\int \eta_{\lambda} |\widetilde{v}_{k, 2}|^2 dv_X dy_{\lambda, 1} \cdots dy_{\lambda, n}. 
\] 
Then (\ref{equation:7}) shows 
\begin{align*}
& \sum_{\lambda} \int \eta_{\lambda}( 2n|\widetilde{v}_{k, 1}|^2  
+ (2n+4)  |\widetilde{v}_{k, 2}|^2) dv_X dy_{\lambda, 1} \cdots dy_{\lambda, n} 
\leq
\sum_{\lambda}\int \eta_{\lambda} L \widetilde{v}_{k} \overline{\widetilde{v}_{k}} dv_X dy_{\lambda, 1}\cdots dy_{\lambda, n}  \\
\leq & (1 + \varepsilon) 2n \sum_{\lambda}\int \eta_{\lambda}(|\widetilde{v}_{k, 1}|^2 + |\widetilde{v}_{k, 2}|^2) dv_X dy_{\lambda, 1}\cdots dy_{\lambda, n}, 
\end{align*}
and 
\[
4 \sum_{\lambda}\int \eta_{\lambda}|\widetilde{v}_{k, 2}|^2 dv_X dy_{\lambda, 1}\cdots dy_{\lambda, n} \leq \varepsilon 2n \sum_{\lambda}\int \eta_{\lambda}|\widetilde{v_k}|^2 dv_X dy_{\lambda, 1}\cdots dy_{\lambda, n}  
\]
for sufficiently large $k$.  
The condition $(\star 1)$ gives $\sum_{\lambda}\int \eta_{\lambda}|\widetilde{v_k}|^2 dv_X dy_{\lambda, 1} \cdots dy_{\lambda, n} = O(1)$, 
and $\sum_{\lambda}\int \eta_{\lambda}|\widetilde{v}_{k, 2}|^2 dv_X dy_{\lambda, 1} \cdots dy_{\lambda, n} = O(1)\varepsilon$.  
We have 
\begin{align*}
& k^{n/2}\| f_k - (\Pi^*a_k s^{k})\|^2_{V'_{k}, h^k} \leq k^{n/2}\| \chi_k (f_k - (\Pi^*a_k s^{k}))\|^2_{h^k}  \\
\lesssim & k^{n/2} \sum_{\lambda}\int \eta_{\lambda}|\chi_k (u_{k}-a_k(x))|^2 e^{-2k|y|^2} dv_X dy_{\lambda, 1} \cdots dy_{\lambda, n} \\
\lesssim & k^{n/2} \sum_{\lambda} \int \eta_{\lambda}(|\chi_k u_{k} - a_k(x)|^2  + (1-\chi_k)^2 |a_k(x)|^2)e^{-2k|y|^2} dv_X dy_{\lambda, 1} \cdots dy_{n, \lambda}  \\
\lesssim & \sum_{\lambda}\int \eta_{\lambda}|\widetilde{v}_{k, 2}|^2 dv_X dy_{\lambda, 1} \cdots dy_{\lambda, n} + \int_{X} |a_k(x)|^2 dv_{X} \int_{\sqrt{k}r_k/2}^{\infty} e^{-2|y|^2} dy_1 \cdots dy_n  \\
=& O(1)\varepsilon  
\end{align*}
for sufficiently large $k$.  
In the last equality, we used 
\begin{align*}
& \int_{X} |a_k(x)|^2 dv_X = \sum_{\lambda} \left(\frac{2}{\pi} \right)^{n/2}\int \eta_{\lambda}|a_k(x)|^2 e^{-2|y|^2}dy_1 \cdots dy_n dv_X \\
\leq & \sum_{\lambda} \int \eta_{\lambda} |\widetilde{v}_{k}|^2 dy_{\lambda, 1}\cdots dy_{\lambda, n} dv_{X} =O(1). 
\end{align*}
This completes the proof of Theorem~\ref{theorem:3} in the case where $\liminf_{k \to \infty}\int_X |f_{k}|^2dv_X > 0$.  
\begin{lemma}\label{lemma:6}
Assume that $M$ is a Stein manifold and $Ric(\omega) \geq -C \omega$ for some $C > 0$.  
Let $f \in C^{\infty}(X)$ such that $\int_X |f|^2 dv_X \neq 0$.  
We take a holomorphic sections $F_k \in \Gamma(M, L^k)$ as in Proposition~\ref{proposition:-2}.  
Let $(r_k)_{k \in \mathbb{N}}$ be a sequence of positive numbers as in Theorem~\ref{theorem:3}.  
Then $(F_k|_{V_k})_{k \in \mathbb{N}}$ satisfies the condition $(\star)$.  
\end{lemma}
\begin{proof}
We use the same notation as in the proof of Proposition~\ref{proposition:-2}. 
That is,  
$U$ is a sufficiently small neighborhood of $X$, 
$\widetilde{f} \in C^{\infty}(U)$ is an almost holomorphic extension of $f$ along $X$, 
and $\xi \in C^{\infty}(U)$ is an almost holomorphic extension of $\frac{\zeta}{s_0}$ along $X$.  
Also $\kappa \in C^{\infty}_0(U)$ is a smooth function such that $\kappa = 1$ on a neighborhood of $X$.  
We define $\alpha_k = \kappa \widetilde{f} s^k$.  
Let 
$\beta_k = b_k s^k$ which is a solution of $\overline{\partial} \beta_k = \overline{\partial}\alpha_k$ such that $\|\beta_k\|^2 = O(k^{-m})$ for any $m \in \mathbb{N}$.  
Then $F_k = \alpha_k -\beta_k = (\kappa \widetilde{f} -b_k)s^k$.  
As $U$ is sufficiently small neighborhood of $X$, we have 
\[
\|\alpha_k\|^2_{h^k} = \int |\kappa \widetilde{f}|^2e^{-k \varphi} \frac{\omega^n}{n!} \lesssim \int \kappa^2 e^{-kd_{X}(z)^2} \frac{\omega^n}{n!} = O(k^{-n/2}).  
\]
Hence $\|F_{k}\|^2_{V_k} = \|\alpha_k - \beta_{k}\|^2_{V_k} \lesssim \|\alpha_k\|^2 + \|\beta_k\|^2 \lesssim k^{-n/2}$, 
and $(F_k|_{V_k})_{k \in \mathbb{N}}$ satisfies 
the condition $(\star1)$.  

We have  $\overline{\partial} b_k = \overline{\partial} (\kappa \widetilde{f}) + k(\kappa \widetilde{f} -b_k)\frac{\overline{\partial} \xi}{\xi}$, 
and  
\begin{align*}
\int_{V_k} |\overline{\partial} b_k|^2_{\omega} e^{-k\varphi} \frac{\omega^n}{n!} 
\lesssim \int_{V_k} (|\widetilde{f}\overline{\partial} \kappa|^2_{\omega} + |\kappa \overline{\partial}\widetilde{f}|^2_{\omega} + k^2 |\widetilde{f} \overline{\partial} \xi|^2_{\omega} + k^2 |b_k \overline{\partial} \xi|^2_{\omega})e^{-k\varphi} \frac{\omega^n}{n!} 
=O(k^{-m})
\end{align*}
for any $m \in \mathbb{N}$ since $\int_{\mathbb{R}^n} |y|^{2m-n}e^{-2k|y|^2}dy_1 \cdots dy_n = O(k^{-m})$ and $\|\beta_k\|^2 = O(k^{-m-2})$.  
The Cauchy-Schwarz inequality implies 
\begin{align*}
& \int \kappa^2  e^{-k\varphi} \partial b_k \wedge \overline{\partial} \overline{b_k} \wedge \omega^{n-1} 
= -\int b_k \partial \left(\kappa^2 e^{-k\varphi}\overline{\partial} \overline{b_k}  \right) \wedge \omega^{n-1} \\
= & -\int b_k e^{-k\varphi}\left(2 \kappa \partial \kappa \wedge \overline{\partial b_k} - k \kappa^2 \partial \varphi \wedge \overline{\partial b_k} + \kappa^2  \partial \overline{\partial b_k} \right)\wedge \omega^{n-1} \\
\leq & \frac{1}{4} \int \kappa^2 e^{-k\varphi} \partial b_k \wedge \overline{\partial b_k} \wedge \omega^{n-1} + 4 \int |b_k|^2 e^{-k\varphi} \partial \kappa \wedge \overline{\partial \kappa} \wedge \omega^{n-1} + \frac{1}{4} \int \kappa^2 e^{-k\varphi} \partial b_k \wedge \overline{\partial b_k} \wedge \omega^{n-1}\\
& +  k^2 \int \kappa^2 e^{-k\varphi} |b_k|^2 \partial \varphi \wedge \overline{\partial \varphi} \wedge \omega^{n-1} 
+ C\int \kappa^2 |b_k|^2 e^{-k\varphi}\omega^n + C \int \kappa^2  |\partial \overline{\partial b_k}|^2_{\omega} e^{-k\varphi} \omega^{n} 
\end{align*}
for some $C > 0$.  
Hence 
\begin{align*}
& \int_{V_k} |\partial b_k|^2 e^{-k\varphi} \frac{\omega^n}{n!} 
\lesssim \int \kappa^2 e^{-k\varphi} \partial b_k \wedge \overline{\partial b_k} \wedge \omega^{n-1} 
\lesssim 
k^2\|\beta_k\|^2_{h^k} + \int \kappa^2  |\partial \overline{\partial b_k}|^2_{\omega} e^{-k\varphi} \omega^{n}.      
\end{align*}
The last term is estimated by 
\[
\int \kappa^2 |\partial \overline{\partial b_k}|^2_{\omega} e^{-k\varphi} \omega^n 
\lesssim \int \left(|\partial \overline{\partial}(\kappa \widetilde{f})|^2_{\omega} + k^2 |\partial ((\kappa \widetilde{f} - b_k)\overline{\partial} \xi)|^2_{\omega}\right) e^{-k\varphi} \omega^n = O(k^{-m}).  
\]
This shows that $\int_{V_k} |\partial b_k|^2_{\omega} e^{-k\varphi} \frac{\omega^n}{n!} = O(k^{-m})$,  and 
$\int_{V_k} |d b_k|^2_{\omega} e^{-k\varphi} \frac{\omega^n}{n!} = O(k^{-m})$.  
We have 
\[
\int_{V_k} |d(\kappa \widetilde{f})|^2_{\omega} e^{-k\varphi} \lesssim \int_{V_k} e^{-k\varphi} = O(k^{-n/2}) 
\]
and $\int_{V_k} |d(\kappa \widetilde{f} - b_k)|^2_{\omega} e^{-k\varphi} = O(k^{-m}) + O(k^{-n/2}) = O(k^{-n/2})$.  
Theorem~\ref{theorem:2} shows  \\
$\liminf_{k \to \infty}\left(\frac{2k}{\pi}\right)^{n/2}\|F_k\|^2_{V_k, h^k} \geq \int_X |f|^2 dv_X > 0$.   
Hence 
\[
\lim_{k \to \infty}\frac{\int_{V_k} |d(\kappa \widetilde{f} - b_k)|^2_{\omega} e^{-k\varphi}\frac{\omega^n}{n!}}{k\|F_k\|^2_{V_k, h^k}} = 0.  
\]
Let $E$ be any smooth vector field on $X$. 
Since $|E^h (\kappa \widetilde{f} -b_k)|^2 \lesssim |d(\kappa \widetilde{f} - b_k)|^2_{\omega}$ on $V_k$, 
it follows that $(F_k|_{V_k})_{k \in \mathbb{N}}$ satisfies $(\star 2)$.  
\end{proof}

\begin{lemma}\label{lemma:7}
Under the same assumption as in Theorem~\ref{theorem:3}, there exists $c > 0$ such that 
$\liminf_{k \to \infty} \int_X |f_k -c\zeta^k|^2_{h^k}dv_X > 0$.   
\end{lemma}
\begin{proof}
Let $C > 0$ such that $\|f_k\|_{V_k, h^k}^2 \leq C k^{-n/2}$.  
Let $c$ be a large positive number which will be determined later.  
Assume that $\liminf_{k \to \infty} \int_X |f_k -c \zeta^k|_{h^k}^2 dv_{X} = 0$.  
Then there exists a sequence of strictly increasing positive numbers $(j_k)_{k \in \mathbb{N}}$ such that 
$\lim_{k \to \infty} \int_X |f_{j_k} - c \zeta^{j_k}|^2_{h^{j_k}} dv_X = 0$.  
By the remark after the proof of Theorem~\ref{theorem:2}, we have 
\begin{align*}
& \liminf_{k \to \infty}\left(\frac{2j_k}{\pi}\right)^{n/2} \|f_{j_k}\|^2_{V_{j_{k}}, h^{j_k}} \geq \liminf_{k \to \infty} \int_X |f_{j_k}|^2_{h^{j_k}}dv_X\\ 
\geq & \liminf_{k \to \infty} \int_X \left(\frac{1}{2}c^2 - |f_{j_k} - c \zeta^{j_k}|^2_{h^{j_k}}  \right) dv_X 
\geq \frac{c^2}{2} \int_X dv_X.   
\end{align*}
Hence, if we choose $c > 0$ such that $C \left(\frac{2}{\pi}\right)^{n/2} < \frac{c^2}{2} \int_X dv_X$, we have $\liminf_{k \to \infty} \int_X |f_k -c\zeta^k|_{h^{k}}dv_X > 0$.  
\end{proof}

We can now finish the proof of Theorem~\ref{theorem:3}.  
Let $(f_{k})_{k \in \mathbb{N}}$ be a sequence of holomorphic sections which satisfies $(\star)$.  
Since $V_k$ is sufficiently small neighborhood of $X$, we may assume that $M$ is Stein and 
$\mathrm{Ric}(\omega) \geq -C\omega$ for some $C > 0$.  
By Lemma~\ref{lemma:7}, there exists $c > 0$ such that $\liminf_{k \to \infty} \int_X |f_k-c \zeta^k|^2_{h^k}dv_X > 0$.  
By Proposition~\ref{proposition:2}, 
there exists a sequence of holomorphic sections $(F_k)_{k \in \mathbb{N}}$ such that $\lim_{k \to \infty}\frac{F_k}{\zeta^k} = -c$ uniformly on $X$.  
We have $(F_k)_{k \in \mathbb{N}}$ satisfies $(\star)$ by Lemma~\ref{lemma:6}.  
Then 
$\int_{X} \liminf_{k \to \infty} |F_k(x)|^2_{h^k}dv_X > 0$, and $\int_X \liminf_{k \to \infty} |f_k + F_k|^2_{h^k} dv_X > 0$.  
By the first part of the proof, there exist sequences $(a_k)_{k \in \mathbb{N}}$, $(\tilde{a}_{k})_{k \in \mathbb{N}}$ of smooth functions on $X$ such that $\lim_{k \to \infty} k^{n/2}\|F_k - \Pi^* a_k\|^2_{V_k', h^k} = \lim_{k \to \infty}k^{n/2}\|f_k + F_k -\Pi^* \tilde{a}_{k}\|^2_{V_k', h^k} = 0$.  
Hence 
\begin{align*}
\lim_{k \to \infty} k^{n/2} \|f_k - \Pi^*(\tilde{a}_k-a_k)\|^2_{V_k', h^k} 
\leq \lim_{k \to \infty} 2k^{n/2} (\|f_k + F_k -\Pi^*\tilde{a}\|^2_{V_k', h^k} + \|F_k - \Pi^* a_k\|^2_{V_k', h^k}) = 0.  
\end{align*}
This completes the proof of Theorem~\ref{theorem:3}.  \qed
 
\begin{remark}
Theorem~\ref{theorem:3} does not hold in the absence of $(\star 2)$.  
Let $c > 0$ be a constant such that the Bergman kernel satisfies 
$|K_{k}(z, w)| \lesssim k^n e^{-c \sqrt{k}  d_{M}(z, w)}$ (see \cite{Ma-Ma2}).  
Here $| \cdot |$ is the norm of $L^k \boxtimes \overline{L}^k$.    
Let $r_k = 4\frac{n\log k}{c\sqrt{k}}$ and $p_k \in M$ such that $d_{X}(p_k) = \frac{n\log k}{c\sqrt{k}}$.  
Put $V_k = \{z \in M\, |\, d_{X}(z) < r_k\}$ and $V'_k = \{z \in M\, |\, d_{X}(z) < r_k/2\}$.  
Let $u_{k} \in L^k$ be a unit element of $L^k$ at $p_k$. 
Let $\mathrm{pr}_{1}, \mathrm{pr}_{2}: X \times X \to X$ be the projection maps.  
Let $f_{k} \in \Gamma(M, L^k)$ such that $k^{-3n/4}K(z, p_{k}) = \mathrm{pr_{1}}^*f_{k}(z) \otimes \mathrm{pr_2}^*\overline{u_{k}}$.  
Then $k^{n/2}\|f_{k}\|^2_{h^k} = k^{-n}|K_{k}(p_k, p_k)| \sim \pi^{-n}$ as $k \to \infty$.  
Since $k^{-3n/4}|K_{k}(z, p_{k})| = O(k^{-3n/4})$ for $z \in M \setminus V_{k}'$, we have $k^{n/2}\|f_{k}\|_{V_{k}', h^k}^2 \sim \pi^n$ as $k \to \infty$.  
Assume that there exists a sequence $(a_k)_{k \in \mathbb{N}}$ of smooth functions on $X$ such that  $\lim_{k \to \infty}k^{n/2}\|f_k - \Pi^* a_k s^k\|^2_{V'_k} = 0$.  
Let $V_k'' = \{z \in M\, |\, d_{X}(z) < \frac{1}{c\sqrt{k}}\}$.  
Then $d_M(z, p_{k}) \geq  \frac{n\log k-1}{c\sqrt{k}}$ and $k^{-3n/4}|K_k(z, p_k)| = O(k^{-3n/4})$ for $z \in V''_k$.  
Hence $\lim_{k \to \infty}k^{n/2} \|f_k\|^2_{V''_{k}, h^k} = 0$, and 
\[
k^{n/2}\|\Pi^{*}a_k s^k\|^2_{V''_{k}, h^k} \lesssim k^{n/2}(\|f_k-\Pi^*a_k s^k\|^2_{V'_{k}, h^k} + \|f_k\|^2_{V''_k, h^k}) \to 0 \quad (k \to \infty).
\]
Since $k^{n/2}\|\Pi^* a_k s^k\|_{V'_k, h^k} \lesssim k^{n/2}\|\Pi^*a_k s^k\|^2_{V''_{k}, h^k}$, we have 
\[
\lim_{k \to \infty} k^{n/2}\|\Pi^*a_k s^k\|_{V_k', h^k} = 0. 
\] 
Then $k^{n/2}\|f_k\|^2_{V'_k, h^k} \lesssim k^{n/2}(\|f_k - \Pi^* a_k\|^2_{V'_k, h^k} + \|\Pi^* a_k s^k\|^2_{V'_k, h^k}) \to 0$ ($k \to \infty$), and this is a contradiction.  
\end{remark}

\section{Asymptotic expansion of $s_{f, k}$}\label{section:7}
%In \cite{Bor-Pau-Uri}, \cite{Deb-Pao}, the asymptotic expansion of $s_{f, k}$ is proved by using the theory of Fourier integral operators of the Hermite type, developed by \cite{Bou-Gui}.  
%In \cite{Pao}, the asymptotic formula is also proved using the scaling asymptotics of Szeg\"o kernels (\cite{Shi-Zel}).  
%On the other hand,  the asymptotic expansion of $\|s_{f, k}\|^2$ is proved by using the asymptotic expansion of generalized Bergman kernel in a symplectic setting.  
%Here, we compute the asymptotic expansion of $s_{f, k}$ using the same method as in \cite{Ioo}.  
%Our proof is similar to that of Theorem~1 of \cite{Pao}.  
The asymptotic expansions of $s_{f, k}$ and $\|s_{f, k}\|^2$ are proved by several approaches.  
In \cite{Bor-Pau-Uri} and \cite{Deb-Pao}, it is shown using the theory of Fourier integral operators; in \cite{Pao}, through the asymptotic expansion of the Szeg\"o kernel; and in \cite{Ioo}, via the asymptotic expansion of the Bergman kernel.
Here, we show the asymptotic expansion of $s_{f, k}$ and $\|s_{f, k}\|^2$ using the expansion formula of Bergman kernel.  
Our proof is similar to that of Theorem~1 in \cite{Pao}.  

We assume that $M$ is compact.  
We use the same notation as in Section~\ref{section:2}.  
That is, $U$ is a sufficiently small neighborhood of $X$ and $s$ is a smooth section of $L$ on $U$ such that $|s|_{h} = 1$ and $\overline{\partial} s$ vanishes to any order on $X$.  
Let $p \in X$.   
Let $V \subset U$ be a sufficiently small neighborhood of $p$.  
We take a smooth coordinate system $(x, y) = (x_1, \ldots, x_n, y_1, \ldots, y_n)$ around $p$ as in Section~\ref{section:2}.  
Furthermore, by taking a linear transformation of $(x_1, \ldots, x_n)$, we may assume that $J \frac{\partial}{\partial x_j} = \frac{\partial}{\partial y_j}$ at $p$ for $j=1, \ldots, n$. 
Let $V_{\mathbb{R}^{2n}}$ be an open neighborhood of $0$ in $\mathbb{R}^{2n}$ and 
let $\Xi: V_{\mathbb{R}^{2n}} \to V$ be a diffeomorphism which corresponds to the smooth coordinate system $(x, y)$.  
In the previous sections, we have not made a distinction between $g$ and $g \circ \Xi$ for a function $g$ on $V$. 
However, from now on, we will distinguish between the two since we need to treat several coordinate systems in the proof of the following theorem.  
\begin{theorem}\label{theorem:4}
There exist polynomials $c_{p, j}(y)$ in $y \in \mathbb{R}^n$ of degree $\leq j$ whose coefficients depend smoothly on $p \in X$ such that  
$c_{p, 2j+1} (0) = 0$ ($j = 0, 1, \ldots$), and the following holds:  
\begin{itemize}
\item[1.]
\[
\frac{s_{f, k}}{s^k} \circ \Xi (0, \ldots, 0, \frac{y_1}{\sqrt{k}}, \ldots, \frac{y_n}{\sqrt{k}})
= f(p) + \sum_{j=1}^N \frac{c_{p, j}(y)}{k^{j/2}} + o\left(\frac{1}{k^{N/2}}\right).  
\]
for any $N \in \mathbb{N}$.  
The convergence of the expansion is uniform for $y \in \mathbb{R}^n$ on any bounded set and $p \in X$.  
\item[2.] 
\[
 \left(\frac{\pi}{2k}\right)^{n/2}\|s_{f, k}\|^2_{h^k} = \int_{X}|f|^2 dv_{X} + \sum_{j=1}^{N} \frac{1}{k^{j}}\int_{X}  c_{p, 2j}(0) \overline{f(p)}dv_{X}(p) + o\left(\frac{1}{k^{N}}\right) 
\]
for any $N \in \mathbb{N}$
\end{itemize}
\end{theorem}
\begin{proof}
Let $V_{\mathbb{C}^n}$ be a small open neighborhood of $0$ in $\mathbb{C}^n$.  
By shrinking  $V \subset M$ if necessary, we take a biholomorphic map $\Upsilon: V_{\mathbb{C}^n} \to V$ such that $\Upsilon(0) = p$ and $x_{j} \circ \Xi^{-1} \circ \Upsilon (z) = \mathrm{Re}\, z_j + O(|z|^2)$, 
$y_j \circ \Xi^{-1}\circ \Upsilon (z) = \mathrm{Im}\,z_j +O(|z|^2)$ for $z=(z_1, \ldots, z_n) \in V_{\mathbb{C}^n}$.  
We take a holomorphic section $t \in \Gamma(V, L)$ such that $t(p) = s(p)$.  
Then $\phi := -\log |t|^2$ is a smooth plurisubharmonic function on $V$ 
such that $\frac{i}{2} \partial \overline{\partial} \phi = \omega$.  
Let $\psi(z, \bar{w})$ ($z, w \in V_{\mathbb{C}^n}$) be a smooth function such that $\overline{\psi(z, \overline{w})} = \psi(\overline{w}, z)$, 
$\psi (z, \bar{z}) = \phi \circ \Upsilon(z)$, and $\overline{\partial} \psi$ vanishes to infinite order at $(z, \overline{z})$ ($z \in V_{\mathbb{C}^n}$).  
The Bergman kernel $K_{k}$ is a smooth section of $L^k\boxtimes \overline{L}^k$ over $M \times M$.  
Let $\pi_{j}: M \times M \to M$ be the $i$-th projection ($i = 1, 2$).  
We denote the function 
$K_{k}/\pi_1^* t^k \pi_2^{*} \overline{t^k}$ by $K_k/t^k \overline{t^k}$ for simplicity.  
Then the Bergman kernel has the following expansion formula (see e.g., \cite{Ber-Ber-Sjo}): 
\begin{align*}
& \left| \frac{K_k(\Upsilon(z), \Upsilon(w))}{t^{k}(\Upsilon(z))\overline{t^k(\Upsilon(w))}} - \frac{k^n}{\pi^n} e^{k\psi(z, \bar{w})}\left(1 + \frac{b_1(z, \bar{w})}{k} + \cdots + \frac{b_{\lceil N/2 \rceil}(z, \bar{w})}{k^{\lceil N/2 \rceil}}\right) \right| \\
\leq & e^{k\frac{\phi \circ \Upsilon(z)}{2}+ k\frac{\phi\circ\Upsilon(w)}{2}}O(\frac{1}{k^{\lceil N/2 \rceil +1-n}}), 
\end{align*}
on $V_{\mathbb{C}^n} \times V_{\mathbb{C}^n}$, 
where $b_j$ is a smooth function determined by the derivatives of $\psi$.  
Let $c(z) = \frac{s(\Upsilon(z))}{t(\Upsilon(z))}$ be a smooth function on $V_{\mathbb{C}^n}$.  
Then $c(0) = 1$, and $\overline{\partial} c$ vanishes to any order on $\Upsilon^{-1}(X)$.  
Let $\varphi = -\log |s|^2$.  
Then $\varphi \circ \Upsilon(z) = \phi \circ \Upsilon(z) - 2 \log |c(z)|$,  
Put $\tilde{\psi}(z, \bar{w}) = \psi(z, \bar{w}) - \log c(z) - \log \overline{c(w)}$.  
Here we take the branch of the logarithm such that $\log 1 = 0$.  
We have 
\begin{align*}
& \left| \frac{K_k(\Upsilon(z), \Upsilon(w))}{s^{k}(\Upsilon(z))\overline{s^k(\Upsilon(w))}} - \frac{k^n}{\pi^n}e^{k\tilde{\psi}(z, \bar{w})}\left(1 + \frac{b_1(z, \bar{w})}{k} + \cdots + \frac{b_{\lceil N/2 \rceil} (z, \bar{w})}{k^{\lceil N/2 \rceil}}\right) \right| \\
\leq & e^{k\frac{\varphi\circ \Upsilon(z)}{2}+ k\frac{\varphi\circ \Upsilon(w)}{2}}O(\frac{1}{k^{\lceil N/2 \rceil +1-n}}).   
\end{align*}
Since 
$\varphi \circ \Xi(x, y) = 2|y|^2 + O(|y|^3)$, we have 
$\varphi \circ \Upsilon(z) = 2|\mathrm{Im}\, z|^2 + O(|z|^3)$. 
Since $\overline{\partial} \tilde{\psi}$ vanishes to infinite order at $(0, 0)$ and $\tilde{\psi}(z, \bar{z}) = \varphi \circ \Upsilon(z)$, 
the Taylor expansion shows that 
$
\tilde{\psi}(z, \bar{w}) = \frac{1}{2}(-z^2 - \bar{w}^2 +2z\bar{w})  + O(|z|^3+|w|^{3}).   
$
We put 
\begin{align*}
\psi_{\Xi}((x, y), (x', y')) & = \tilde{\psi}(\Upsilon^{-1}\circ \Xi(x, y), \overline{\Upsilon^{-1} \circ \Xi (x', y')}), \\
b_{\Xi, j}((x, y), (x', y')) & = b_j(\Upsilon^{-1}\circ \Xi(x, y), \overline{\Upsilon^{-1} \circ \Xi (x', y')})  
\end{align*}
for $(x, y) \in V_{\mathbb{R}^n}$.
We have 
\begin{align}\label{equation:8}
k \psi_{\Xi}((0, \frac{y}{\sqrt{k}}), (x, 0)) 
= \frac{1}{2} y^2 - \frac{k}{2}x^2 + \sqrt{-1}\sqrt{k} xy + kR(\frac{y}{\sqrt{k}}, x) + k O(|\frac{|y|}{\sqrt{k}}|^{N+3} + |x|^{N+3}).  
\end{align}
Here $R(y, x)$ is a polynomial of $3 \leq \deg R \leq N+2$, and coefficients of $R$ depends smoothly on $p$. 
Let $\iota:\mathbb{R}^n \to \mathbb{R}^{2n}$ be a map such that $\iota(x) = (x, 0)$.  
Then $\Xi\circ\iota$ is a diffeomorphism from $\iota^{-1}(V_{\mathbb{R}^{2n}})$ to $X \cap V$.  
We consider $\Xi\circ\iota$ as a map from $\iota^{-1}(V_{\mathbb{R}^{2n}})$ to $X$, and 
we expand the pull back of 
a smooth density $f dv_{X}$ by $\Xi\circ\iota$ as 
$
(\Xi\circ\iota)^{*}(fdv_{X}) = \left(f(p) + \sum_{1 \leq |I| \leq N} a_{I} x^{I} + O(|x|^{N+1}) \right) dx_1 \cdots dx_n.  
$
Here $a_I \in \mathbb{C}$ is determined by the derivatives of $f$ and $g_{\omega}(\frac{\partial}{\partial x_i}, \frac{\partial}{\partial x_j})$ at $p$.  
It is known that the off-diagonal asymptotic of the Bergman kernel satisfies 
$|K_{k}(Z, Z')| \lesssim k^n e^{-c \sqrt{k}  d_{M}(Z, Z')}$ for some $c > 0$, where $d_M(Z, Z')$ is the distance between $Z \in M$ and $Z' \in M$ (see Theorem~0.1 of \cite{Ma-Ma2}).  
Here $| \cdot |$ is the norm of $L^k \boxtimes \overline{L}^k$ induced by the Hermitian metric $h$ of $L$.  
Let $c' > 0$ be a sufficiently small number. 
Let $\eta:\mathbb{R} \to \mathbb{R}$ be a smooth function such that $\eta(r) = 1$ if $0 \leq r \leq 1$ and $\eta(r) = 0$ if $r \geq 2$.  
We put $\eta_{k, N}(x) = \eta \left(\frac{c'\sqrt{k}}{(\lceil N/2 \rceil+n/2+1)\log k}|x|\right)$.  
Then 
\[
\left|
(1-\eta_{k, N}(x)) \frac{K_k(\Xi(0, \frac{y}{\sqrt{k}}), \Xi(x, 0))}{s^k(\Xi(0, \frac{y}{\sqrt{k}}))\overline{s^k(\Xi(x, 0))}}\right| 
= e^{\frac{k}{2} \varphi \circ \Xi(0, \frac{y}{\sqrt{k}})} O\left(\frac{1}{k^{\lceil N/2 \rceil - n/2 + 1}}\right), 
\]
and we have 
\begin{align*}
&s_{f, k} \circ \Xi(0, \ldots, 0, \frac{y_1}{\sqrt{k}}, \ldots, \frac{y_n}{\sqrt{k}}) \\
= & \left(\frac{\pi}{2k}\right)^{n/2} \int_{x \in \mathbb{R}^n} \eta_{k, N}(x) \frac{K_{k}(\Xi(0, \frac{y}{\sqrt{k}}), \Xi(x, 0))}{\overline{s^k(\Xi(x, 0))}}  \left(f(0) + \sum_{1 \leq |I| \leq N} a_{I} x^{I} + O(|x|^{N+1}) \right) dx_1 \cdots dx_n \\
& +  s^k \circ \Xi(0, \frac{y}{\sqrt{k}}) e^{\frac{k}{2}\varphi\circ \Xi(0, \frac{y}{\sqrt{k}})}O\left(\frac{1}{k^{\lceil N/2 \rceil +1}}\right).  
\end{align*}
For any $y \in \mathbb{R}^n$ on a bounded set and $x \in \iota^{-1}(V_{\mathbb{R}^{2n}})$,  we have 
$k\varphi\circ\Xi(0, \frac{y}{\sqrt{k}}) = 2 |y|^2 + O(\frac{|y|^3}{\sqrt{k}}) = O(1)$, $k \varphi \circ \Xi(x, 0) = 0$, 
and 
\begin{align*}
& \left(\frac{\pi}{2k}\right)^{n/2}
\frac{K_{k}(\Xi(0, \frac{y}{\sqrt{k}}), \Xi(x, 0))} {\overline{s^k \circ \Xi(x, 0)}} 
\left(f(p) + \sum_{1 \leq |I| \leq N} a_{I} x^{I} + O(|x|^{N+1}) \right) dx_1 \cdots dx_n  \\
= &  
s^k\circ \Xi(0, \frac{y}{\sqrt{k}})\left(\frac{k}{2\pi}\right)^{n/2} e^{\frac{y^2}{2} - k \frac{x^2}{2} + \sqrt{-1}\sqrt{k}xy}\left(1+kR(\frac{y}{\sqrt{k}}, x) + kO(\big|\frac{y}{\sqrt{k}}\big|^{N+3} + |x|^{N+3})
 \right)\\
& \times  
\left(1 + \sum_{j=1}^{\lceil N/2 \rceil} \frac{b_{\Xi, j}((0, \frac{y}{\sqrt{k}}), (x, 0)}{k^{j}} \right)  
\left(f(p) + \sum_{1 \leq |I| \leq N} a_{I} x^{I} + O(|x|^{N+1}) \right) dx_1 \cdots dx_n \\
& +  s^k\circ\Xi(0,\frac{y}{\sqrt{k}})O\left(\frac{1}{k^{\lceil N/2 \rceil - n/2 +1}}\right) dx_1\ldots dx_n.  
\end{align*}
By replacing $x$ with $\frac{x}{\sqrt{k}}$, we have 
\begin{align*} 
& \int_{x \in \mathbb{R}^n} \eta_{k, N}(x)
\frac{K_{k}(\Xi(0, \frac{y}{\sqrt{k}}), \Xi(x, 0))} {\overline{s^k \circ \Xi(x, 0)}} 
\left(f(p) + \sum_{1 \leq |I| \leq N} a_{I} x^{I} + O(|x|^{N+1}) \right) dx_1 \cdots dx_n  \\
= & s^k\circ \Xi(0, \frac{y}{\sqrt{k}})\frac{e^{\frac{y^2}{2}}}{(2\pi)^{n/2}} \int_{x \in \mathbb{R}^n} e^{-\frac{x^2}{2}+\sqrt{-1}xy} \left(f(p) + \sum_{j=1}^{N}\frac{Q_{p, j}(y, x)}{k^{j/2}} \right) dx_1\cdots dx_n \\
& +  s^k\circ\Xi(0, \frac{y}{\sqrt{k}}) o\left(\frac{1}{k^{N/2}}\right).   
\end{align*}
for $y \in \mathbb{R}^n$ on a bounded set.  
Here, $Q_{p, j}(y, x)$ is a polynomial of degree $\leq j$ in $y$ and $x$, which is determined by $R, b_{\Xi, j}$ and $a_I$.  
In each term of $Q_{p, 2j+1}(y, x)$, the total degree of $y$ and $x$ is odd.  
For mult-index $I = (i_1, \ldots, i_n)$ ($i_j \geq 0$), we have 
\[
\frac{1}{(2\pi)^{n/2}}\int_{\mathbb{R}^n} x^{I} e^{-\frac{x^2}{2} + \sqrt{-1}xy}dx_1 \ldots dx_n 
= (-\sqrt{-1})^{|I|}\left(\frac{\partial}{\partial y}\right)^{I} e^{-\frac{y^2}{2}}.  
\]
Put 
\[
c_{p, j}(y) = \frac{e^{\frac{\upsilon^2}{2}}}{(2\pi)^{n/2}} \int_{x \in \mathbb{R}^n} e^{-\frac{x^2}{2}+\sqrt{-1}xy} Q_{p, j}(y, x) dx_1\cdots dx_n.   
\]
Then $c_{p, j}(y)$ is a polynomial of degree $\leq j$ in $y$, and $c_{p, 2j+1}(0) = 0$ since the degree of each term of $Q_{p, 2j+1}(0, x)$ is odd.   
The coefficients of $c_{p, j}$ are smooth functions of $p$ determined due to its construction.  
Now we have the asymptotic expansion 
\[
\frac{s_{f, k}}{s^k} \circ \Xi(0, \frac{y}{\sqrt{k}}) = f(p) + \sum_{j=1}^{N} \frac{c_{p, j}(y)}{k^{j/2}} + o\left(\frac{1}{k^{N/2}}\right).  
\]
This convergence is uniform for $y \in \mathbb{R}^n$ on any bounded set and $p \in X$.  
Then we have 
\[
 \left(\frac{\pi}{2k}\right)^{n/2}\|s_{f, k}\|^2_{h^k} = \int_{X} s_{f, k}\overline{f} dv_{X} = \int_{X}|f|^2 dv_{X} + \sum_{j=1}^{N} \frac{1}{k^{j}}\int_{X}  c_{p, 2j}(0) \overline{f(p)}dv_{X}(p) + o\left(\frac{1}{k^{N}}\right).  
\]
\end{proof}

\vspace{5mm}

\par\noindent{\scshape \small
Department of Mathematics, \\
Ochanomizu University,  \\
2-1-1 Otsuka, Bunkyo-ku, Tokyo (Japan) }
\par\noindent{\ttfamily chiba.yusaku@ocha.ac.jp}
\end{document}